\documentclass{amsart}
\usepackage{amsmath,amssymb,enumerate,mathtools}
\newtheorem{theorem}{Theorem}[section]
\newtheorem{claim}{}[theorem]
\newtheorem{lemma}[theorem]{Lemma}

\newtheorem{corollary}[theorem]{Corollary}
\newtheorem{conjecture}[theorem]{Conjecture}
\theoremstyle{definition}
\newtheorem{definition}[theorem]{Definition}

\newenvironment{subproof}[1][\proofname]{%
	\begin{proof}[Subproof:]%
	}{%
	\end{proof}%
}
\newcommand{\floor}[1]{\left\lfloor #1 \right\rfloor}
\newcommand{\ceil}[1]{\left\lceil #1 \right\rceil}

\newcommand{\bF}{\mathbb F}
\newcommand{\bR}{\mathbb R}
\newcommand{\bC}{\mathbb C}

\newcommand{\bN}{\mathbb N}

\newcommand{\cB}{\mathcal{B}}

\newcommand{\cE}{\mathcal{E}}

\newcommand{\cL}{\mathcal{L}}

\newcommand{\cM}{\mathcal{M}}

\newcommand{\cS}{\mathcal{S}}
\newcommand{\cT}{\mathcal{T}}
\newcommand{\cU}{\mathcal{U}}

\newcommand{\bT}{\mathbf{T}}

\newcommand{\nni}{\bN_0}
\newcommand{\posi}{\bN}

\newcommand{\es}{\varnothing}
\newcommand{\para}{\approx}
\DeclareMathOperator{\si}{si}

\DeclareMathOperator{\cl}{cl}

\DeclareMathOperator{\PG}{PG}
\DeclareMathOperator{\DG}{DG}

\DeclareMathOperator{\GF}{GF}
\DeclareMathOperator{\AG}{AG}
\DeclareMathOperator{\FM}{F}
\newcommand{\elem}{\varepsilon}
\newcommand{\del}{\!\setminus\!}
\newcommand{\con}{/}

\usepackage[utf8]{inputenc}
\usepackage{subfig}

\title[Clique minors of matroids]{On the density of matroids omitting a complete-graphic minor}
\author{Peter Nelson}
\author{Sergey Norin}
\author{Fernanda Rivera Omana}
\begin{document}

\begin{abstract}
  We show that, if $M$ is a simple rank-$n$ matroid with no $\ell$-point line minor and no minor isomorphic to the cycle matroid of a $t$-vertex complete graph, then the ratio $|M| / n$ is bounded above by a singly exponential function of $\ell$ and $t$. We also bound this ratio in the special case where $M$ is a frame matroid, obtaining an answer that is within a factor of two of best-possible. 
\end{abstract}
\maketitle

\section{Introduction}

Write $M(K_t)$ for the cycle matroid of the complete graph $K_t$; that is, the matroid whose elements are the edges of $K_t$, and whose circuits are the cycles of $K_t$. We prove the following. 

\begin{theorem}\label{mainbin}
  Let $t,n \in \posi$. If $M$ is a simple rank-$n$ binary matroid with no $M(K_t)$-minor, then $|M| \le 2^{t^4 / 2} \cdot n$. 
\end{theorem}

This is new for all $t \ge 6$. For small $t$, some results already existed. When $t = 3$, we trivially have $|M| \le n$. For $t = 4$, a simple rank-$n$ binary matroid with no $M(K_4)$-minor is the cycle matroid of a series-parallel graph ([\ref{oxley}], Corollary 12.2.14), and so can be shown to satisfy the best-possible bound $|M| \le 2n-1$. For $t = 5$, Kung [\ref{kung87}] showed that $|M| \le 8 n$. For general $t$, the best result was due to Geelen ([\ref{g11}], also see [\ref{oxley}], Theorem 14.10.10), who proved that $|M| \le 2^{2^{3t}}n$. 

The binary matroids are those with no $U_{2,4}$-minor, and in fact Theorem~\ref{mainbin} is a consequence of techniques we develop in the more general setting of matroids that omit some fixed rank-$2$ uniform matroid as a minor. We obtain a slightly weaker upper bound in this expanded setting. 
\begin{theorem}\label{main}
  Let $\ell,n \in \posi$. If $t \in \bN$ is sufficiently large, then every simple rank-$n$ matroid $M$ with no $U_{2,\ell+2}$-minor or $M(K_t)$-minor satisfies $|M| \le \ell^{{(\ell-1)}^2 t^4 \log t} \cdot n$. 
\end{theorem}

The logarithm is natural, as are all that follow. Kung [\ref{kungll}] proved that if $M$ is simple with no $U_{2,\ell+2}$-minor or $M(K_4)$-minor, then $|M| < (6\ell^{\ell-1} + 8\ell) n$, but the previously best known upper bound for general $\ell$ and $t$ had a coefficent doubly exponential in $t$; the aforementioned result of Geelen showed that $|M| \le \ell^{\ell^{3t}}n$. By contrast, our coefficient of $n$ is singly exponential in both $t$ and $\ell$. The requirement that $t$ is large can be dropped in exchange for a factor of a little under $4000$ in the exponent; see Theorem~\ref{mainabs}. 

We also show that a better-than-exponential upper bound on $|M|/n$ is not possible, first giving some notation to facilitate the discussion. The \emph{extremal function} of a nonempty minor-closed class $\cM$ of matroids is the function $h_{\cM} : \nni \to \nni \cup \{\infty\}$ defined so that $h_{\cM}(n)$ is the maximum number of elements in a simple matroid in $\cM$ whose rank is at most $n$. This was first studied by Kung [\ref{kung87}], who called it the \emph{size function}. It was later called the \emph{growth rate function}; our terminology aligns with graph theory. Kung proved that $h_{\cM}$ is everywhere finite if and only if there is some $\ell \in \posi$ for which $U_{2,\ell+2} \notin \cM$. Later, Geelen and Whittle [\ref{gw}] showed that $h_{\cM}(n) = O(n)$ if and only if there exist $\ell,t \in \posi$ for which $U_{2,\ell+2} \notin \cM$ and $M(K_t) \notin \cM$. This makes the hypothesis of Theorem~\ref{main} very natural; it excludes a canonical pair of minors that define a class with linear extremal function. We can restate Theorem~\ref{main} in these terms, this time including a lower bound. 

\begin{theorem}\label{main2}
  Let $\ell \in \posi$, and $q$ be the largest prime power with $q \le \ell$. If $t$ is sufficiently large, and $\cM$ is the class of matroids with no $U_{2,\ell+2}$-minor and no $M(K_t)$-minor, then 
  \[q^{t-3} n + \left(\tfrac{q^{t-3}-1}{q-1}-(t-3)q^{t-3}\right) \le h_{\cM}(n) \le \ell^{{(\ell-1)}^2 t^4 \log t} n\]
  for all sufficiently large $n$. 
\end{theorem}
The exponent in the upper bound can be multiplied by any $\beta$ greater than approximately $0.81$; see Theorem~\ref{maintechasymp}. We leave it in the cleaner form above because the upper bound is likely far from tight. 

Since there is some $k$ for which $\tfrac{\ell+1}{2} \le 2^k \le \ell$, the lower bound is asymptotically not much less than $\left(\tfrac{\ell+1}{2}\right)^t n$, and by standard results on the distribution of primes, is in fact equal to $\ell^{(1-o_{\ell}(1))t}$. While there is still a large gap between the upper and lower bounds, this theorem establishes that when $\ell$ is held constant, the ratio $\frac{h_{\cM}(n)}{n}$ grows singly exponentially with $t$. This is new, even for $\ell = 2$. 

The lower bound comes from a class of matroids called \emph{crowns}, which we define in Section~\ref{crownsection}. Briefly, they are the maximal simple matroids having a $\GF(q)$-representation $A$ with $n$ rows, from which $t$ rows can be removed to obtain a matrix where every column is a standard basis vector or zero. Their construction, using finite fields, necessitates prime powers. We include the exact lower bound because we believe it is correct for large $t$ and $n$; that is, crowns are the extremal examples.

\begin{conjecture}\label{mainconj}
  Let $\ell \in \posi$ with $\ell \ge 2$, and $q$ be the largest prime power with $q \le \ell$. If $t \in \posi$ is sufficiently large, and $\cM$ is the class of matroids with no $U_{2,\ell+2}$- or $M(K_{t+3})$-minor, then $h_{\cM}(n) = q^t n + \left(\tfrac{q^t-1}{q-1}-tq^t\right)$ for all sufficiently large $n$.
\end{conjecture}

It might seem strange to believe that the distribution of prime powers should determine a combinatorial bound in this way, but such behaviour is known to occur in similar problems; see [\ref{gn10}]. Proving this conjecture would require substantial new ideas. A large $t$ is necessary; for $\ell = t = 2$, denser examples than crowns are known [\ref{kung87}], and we show in Lemma~\ref{coupled} that denser examples also exist if $2 < q = \ell = t+2$. 

As we will see, an $n$-crown with no $M(K_{t+3})$-minor or $U_{2,\ell+2}$-minor will be constructed by taking the union of $n-t$ flats of a rank-$n$ projective geometry over $\GF(q)$, where $q \le \ell$. Its density is therefore concentrated in a small (linear) number of very dense flats, making it very unlike a graphic matroid. We now consider matroids that resemble graphs more closely.

\subsection*{Frame Matroids}

A \emph{frame matroid} is a restriction of a matroid $M$ having a basis $B$ such that every element of $M$ is spanned by a subset of $B$ of size at most $2$. These matroids are `graph-like', with $B$ playing the role of a vertex set, and the elements corresponding to edges. They are well-studied for their links to groups, graphs and topology [\ref{cdfp},\ref{dfp},\ref{zas1},\ref{zas94},{\ref{zas99}}], and arise naturally in structure and extremal theory of minor-closed classes [\ref{gnw},\ref{kk}]. 

The frame matroids themselves form a minor-closed class, and the proof of Theorem~\ref{main} makes essential use of them; as a consequence of the techniques used, we prove an analogue of Theorem~\ref{main2} specifically for this class. Unlike the large gap in bounds in Theorem~\ref{main2}, we in fact determine $h_{\cM}(n)$ to within an asymptotic factor of $2$. (Here, $\alpha = 0.319\dotsc$ is an explicit real constant; see Theorem~\ref{thomason}.)

\begin{theorem}\label{mainframeclique}
  Let $\ell,t \in \posi$ with $\ell \ge 2$. If $\cM$ is the class of frame matroids with no $U_{2,\ell+2}$-minor and no $M(K_t)$-minor, then \[(\ell-1)(\alpha + o_t(1)) t \sqrt{\log t} \cdot n \le h_{\cM}(n) \le 2(\ell-1)(\alpha + o_t(1)) t \sqrt{\log t} \cdot n\]
  for all sufficiently large $n$. 
\end{theorem}

The upper bound on $h_{\cM}(n)/n$ that this gives is not much more than linear in $t$; this contrasts dramatically with the lower bound on $h_{\cM}(n)/n$ in Theorem~\ref{main2}, which is exponential in $t$. Theorem~\ref{mainframeclique} is proved with ideas that are essentially graph-theoretic, namely a reduction to the well-understood problem of bounding average degree in graphs with no $K_t$-minor. This is the source of the logarithmic terms and the constant $\alpha$. The class of frame matroids is amenable to such techniques because frame matroids are a qualitatively different setting to that of general matroids, being much closer to graphs. 
We believe that the lower bound in Theorem~\ref{mainframeclique} is probably correct, but the factor of $2$ is inherent in our techiques.

We also consider the class of frame matroids with no $U_{2,\ell+2}$-minor, determining its extremal function exactly. This proof is routine, but, to the best of our knowledge, the result is new. 

\begin{theorem}\label{mainframe}
  Let $\ell \in \posi$. If $\cM$ is the class of frame matroids with no $U_{2,\ell+2}$-minor, then $h_{\cM}(n) = (\ell-1)\binom{n}{2} + n$ for all $n \in \posi$. 
\end{theorem}

When $n \ne 3$, we also determine which matroids attain this extremal function; see Theorem~\ref{extremalframe}. The required background material on frame matroids is given in Section~\ref{prelim}, and Theorems~\ref{mainframeclique} and~\ref{mainframe} will be proved in Section~\ref{framesection}.

Finally, we believe that if the setting of Theorem~\ref{main} is restricted to the matroids representable over a field $\bF$ of characteristic zero, then the answer becomes much more like Theorem~\ref{mainframe}. Specifically, for each such $\bF$, the densest $\bF$-representable matroids with no $M(K_t)$-minor and no $U_{2,\ell+2}$ should be obtained from dense $K_t$-minor-free graphs by a `blow-up' operation using the largest finite multiplicative subgroup $\Gamma$ of $\bF$ for which $|\Gamma| < \ell$. (See the definition of $\Gamma$-frame matroids in Section 2 for a more precise discussion of this construction). We make two conjectures treating the cases of the real and complex numbers. 

\begin{conjecture}\label{Rconj}
  Let $\ell,t \in \posi$ with $\ell \ge 3$. If $\mathcal{M}$ is the class of $\bF$-representable matroids with no $U_{2,\ell+2}$-minor and no $M(K_t)$-minor, then 
    \[h_{\cM}(n) = 2(\alpha + o_{t,\ell}(1)) t \sqrt{\log t} \cdot n.\] 
\end{conjecture}

The factor of two appears because this is the order of the largest multiplicative subgroup of $\bR$. For the complex-numbers, larger subgroups exist, and we obtain a conjectured answer that asymptotically depends on $\ell$ as well as $t$. 

\begin{conjecture}\label{Cconj}
  Let $\ell,t \in \posi$ with $\ell \ge 2$. If $\mathcal{M}$ is the class of $\bC$-representable matroids with no $U_{2,\ell+2}$-minor and no $M(K_t)$-minor, then 
    \[h_{\cM}(n) = (\ell-1)(\alpha + o_{t,\ell}(1)) t \sqrt{\log t} \cdot n.\] 
\end{conjecture}

\subsection*{The proof}

We now give an informal discussion of the proof of Theorem~\ref{main}, both to highlight the new ideas, and to discuss the obstacles to strengthening the result. 
The exponent in the upper bound of Theorem~\ref{main2} is roughly the fourth power of the exponent in the lower bound, but we believe the lower bound to be correct, so such a strengthening should be possible. 
Think of $\ell$ as an absolute constant, and consider some $t \in \posi$. 

Theorem~\ref{main} is proved by first considering a simple matroid $M$ with $|M| > \beta_1 r(M)$, where $\beta_1$ has order $\ell^{t^4 \log t}$. By iterating an inductive argument that exploits minor-minimality, we show that $M$ has a highly structured minor called a \emph{tower}, of rank $\beta_2$, where $\beta_2$ has order $t^2 \sqrt{\log t}$. Towers are defined in Section~\ref{towersection}, and Lemma~\ref{gettower}, which finds a tower in a dense matroid, is the main result of Section~\ref{buildsection}. 

Towers are technical objects, but a rank-$\beta_2$ tower can be partially described by a connected $\beta_2$-vertex digraph $G$. Using $\beta_2 \approx t (t \sqrt{\log t})$ and some Ramsey-like ideas, we can either find a path of length $t$ in $G$, or a frame matroid minor corresponding to a clique of size order $t \sqrt{\log t}$. The path outcome turns out to correspond directly to an $M(K_t)$-minor of $G$, while the clique outcome gives an $M(K_t)$-minor using the ideas from Theorem~\ref{mainframeclique}. This is all done in Section~\ref{exploitsection}.

The proof of the main theorem of [\ref{g11}] used a somewhat similar approach, but obtains a doubly exponential upper bound. The first part of Geelen's proof uses a similar inductive argument to ours (indeed, ours is based on his) to find an object called a `constructed flat' with less structure than a tower, but whose rank is still logarithmic in the initial density constant. It then shows that a constructed flat $F$ contains an $M(K_t)$-minor, where $t$ is logarithmic in the rank of $F$. Taking logarithms twice in this way is what gave the doubly exponential upper bound. Our main contribution is the use of towers rather than constructed flats, as well as the analysis of towers themselves; together, they necessitate taking just a single logarithm in total. 

Thus the proof has two parts: going from a matroid whose density is exponential in $t$ to a tower whose rank is polynomial in $t$, and then showing that such a tower contains $M(K_t)$. Our current method of finding a tower, Lemma~\ref{gettower}, requires a matroid with $|M| > \ell^{\binom{s+1}{2}}r(M)$ to find a tower of height $s$. It is not clear how to do this with our particular techniques, but changing $\ell^{\binom{s+1}{2}}$ to $\ell^{s}$ would change the fourth power in our upper bound to a square. 

Something that might be easier to improve in our proof is the second half: the analysis of towers. Currently we need a tower of rank $t^2 \sqrt{\log t}$ to find an $M(K_t)$-minor, but in fact we know of no example of an $M(K_t)$-minor-free tower with rank more than linear in $t$. If no such tower exists, then again we would halve the exponent, giving a square rather than a fourth power in the upper bound.  

\section{Preliminaries}\label{prelim}

Write $\posi$ for the set of positive integers, and $\nni = \posi \cup \{0\}$. Write $[n]$ for $\{1, \dotsc, n\}$, and $[m,n] = \{m, \dotsc, n\}$ for $m,n \in \nni$.

\subsection*{Matroids}
We use the notation of Oxley [\ref{oxley}]. We also identify a matroid with its ground set in various contexts; if $M$ is a matroid, then $|M|$ denotes $|E(M)|$ and $x \in M$ means $x \in E(M)$. Let $\elem(M)$ denote the number of points in a matroid $M$. In particular, $\elem(M) = |M|$ if and only if $M$ is simple. If $e, f \in M$, we write $e \approx_M f$ if $e$ and $f$ are parallel nonloops of $M$ (we consider equal elements to be parallel). 

For $\ell \in \posi$, write $\cU(\ell)$ for the class of matroids with no $U_{2,\ell+2}$-minor. It is clear that $\cU(1)$ is the (trivial) class of matroids whose simplification is a free matroid, so $\elem(M) = r(M)$ for all $M \in \cU(1)$. For $\ell \ge 2$, the following density bound due to Kung [\ref{kung93}] is ubiquitous in extremal matroid theory. (Kung actually proved the bound $\elem(M) \le \frac{\ell^{r(M)}-1}{\ell-1}$, but we give a cruder version purely to be concise.)

\begin{theorem}[{[\ref{kung93}]}]\label{kung}
  Let $\ell \in \posi$ with $\ell \ge 2$. If $M \in \cU(\ell)$, then $\elem(M) < \ell^{r(M)}$. 
\end{theorem}

\subsection*{Graphs}

Our graphs are not necessarily simple, taking the form $G = (V,E,\iota)$ for arbitrary sets $V$ and $E$, where $\iota$ is an incidence function relating $V$ and $E$.  A \emph{cycle} of $G$ is a $2$-regular connected subgraph of $G$. A \emph{minor} of $G$ is a graph $H$ obtained from $G$ by deleting edges/vertices and contracting edges; we have $E(H) \subseteq E(G)$. Formally, one must decide on the labels for the vertices of $H$ when contracting, but this won't affect us; an \emph{$H$-minor} of $G$ is a minor of $G$ isomorphic to $H$. 

A \emph{theta} in a graph $G$ is a subgraph of $G$ comprising a pair of $u,v$ of distinct vertices of $G$, together with three internally disjoint $uv$-paths of $G$. The graph $G$ and its thetas need not be simple; for example, any three parallel edges form a theta. A theta of $G$ contains precisely three cycles of $G$. 

A \emph{handcuff} of a graph $G$ is a minimal connected subgraph $H$ of $G$ containing exactly two cycles. More explicitly, a handcuff comprises two cycles $C_1$ and $C_2$ and a $C_1C_2$-path $P$, for which either $C_1$ and $C_2$ are vertex-disjoint, or $P$ is trivial and $V(C_1)$ and $V(C_2)$ intersect exactly at the vertex of $P$. Thetas and handcuffs are closely related; a minimal connected subgraph of $G$ containing more than one cycle must be a theta or a handcuff of $G$. 

We need some results from extremal graph theory. For each $t \in \posi$, let $d(t)$ be the infimum of the set of all $c \in \bR$ such that every simple graph $G = (V,E)$ with $|E| > c |V|$ has a $K_t$-minor. (So if $|E| > d(t)|V|$ for any simple graph $G = (V,E)$, then $K_t \le G$). Thomason [\ref{t01}] determined $d(t)$ asymptotically. Here, $\alpha = 0.319 \dotsc$ is the explicit real constant defined by $\alpha =  \tfrac{1}{2}(1-\lambda) \sqrt{\log \lambda^{-1}}$, where $\lambda < 1$ is the unique solution of the equation $\lambda (1 + 2 \log \lambda) = 1$.

\begin{theorem}[{[\ref{t01}]}]\label{thomason}
  $d(t) = (\alpha + o(1)) t \sqrt{\log t}$.
\end{theorem}

For a version of Theorem~\ref{main} that applies for all $t$, we need absolute, and necessarily weaker, bounds on $d(t)$. The lower bound here is well-known, and the upper bound is essentially due to Kostochka (who phrased it in terms of the inverse function of $d$, referred to as $\eta$). 
\begin{theorem}\label{kostochka}
  $t - 2 \le d(t) \le 22 t\sqrt{\log t}$ for all $t \in \bN$. 
\end{theorem}
\begin{proof}
  The result is trivial for $t = 1$, since $d(1) = 0$; assume that $t \ge 2$. The fact that $d(t) \ge t-2$ follows by considering the union of $n-t+2$ copies of $K_{t-1}$ that intersect in $t-2$ common vertices, which has $n$ vertices and $(t-2)n - O(t^2)$ edges, with no $K_t$-minor. 
  
  By the main theorem of [\ref{k82}], for each $c \ge 2$, every simple graph $G$ with $|E(G)| > c |V(G)|$ has a $K_t$-minor for some $t \ge \tfrac{0.064c}{\sqrt{\log c}} + 1$. To show the upper bound, it thus suffices to show that, if $c(t) = 22 t \sqrt{\log t}$, then $\frac{0.064c(t)}{\sqrt{\log c(t)}} + 1 > t-1$ for all $t \ge 2$. If $t \ge 22^2$, then using $\log(t) \le \sqrt{t}$, we have $c(t) \le t^{7/4}$ and so $\frac{0.064c(t)}{\sqrt{\log(c(t))}} \ge \frac{1.408 t \sqrt{\log t}}{\sqrt{7/4} \sqrt{\log t}} > t$. One can verify the smaller values of $t$ with a numerical computation. 
\end{proof}

The following easy lemma will be used for technical reasons to relate the definition of $d(t)$ to matroid extremal functions. 

\begin{lemma}\label{dlimit}
  Let $t \in \posi$. For all $d' \in \bR$ with $d' < d(t)$, there exists $n_0 \in \posi$ such that, for all $n \ge n_0$, there is a simple graph $G_n$ on $n$ vertices, with no $K_t$-minor, for which $|E(G_n)| > d' n$. 
\end{lemma}
\begin{proof}
  By the definition of $d(t)$, there is a simple graph $H = (V,E,\iota)$ with no $K_t$-minor, for which $|E| > d' |V|$; let $k = |V|$ and $s = |E|$, so $s/k > d'$. Let $n_0 = \ceil{\frac{s+1}{s/k - d'}}$. 
  For each $n \ge n_0$, let $G_n$ be the graph on $n$ vertices that is the direct sum of $\floor{n/k}$ copies of $H$ and some isolated vertices. Then $G_n$ has no $K_t$-minor, and $|E(G_n)| = \floor{n/k}s \ge (n/k - 1)s > d' n$, where the last inequality uses $n(s/k-d') \ge n_0(s/k-d') > s$. 
\end{proof}

\subsection*{Frame Matroids}

We now discuss the basic material we will need concerning frame matroids. We include some proofs, but the material here is standard; a reader familiar with frame matroids, biased graphs and Dowling geometries can skip the remainder of this section. 

A matroid $M$ is \emph{framed} by a set $B$ if $B$ is a basis of $M$, and every element of $E(M)$ is spanned by a subset of $B$ of size at most $2$. A \emph{frame matroid} is a restriction of a matroid framed by some set $B$. These matroids were defined and studied by Zaslavsky [\ref{zas99}]. Such matroids are `graph-like', with $B$ playing the role of a vertex set, and the element of $M$ being edges. 

Frame matroids are combinatorially represented by `biased graphs'. A \emph{biased graph} is a pair $\Omega = (G,\cB)$, where $G$ is a (not necessarily simple) graph, and $\cB$ is a set of cycles of $G$ called \emph{balanced} cycles, such that no theta of $G$ contains precisely two balanced cycles. For example, $\cB$ could be empty, or the collection of all cycles of $G$, or the collection of all even cycles of $G$. We say $\Omega$ is \emph{connected} if $G$ is connected, and that $\Omega$ is \emph{balanced} if every cycle of $G$ is balanced. 

Given a biased graph $\Omega = (G,\cB)$, let $\cU$ denote the collection of edge sets of handcuffs and thetas of $G$ that do not contain a cycle in $\cB$. The set $\cB \cup \cU$ is the collection of circuits of a matroid on $E(G)$; call this matroid $\FM(\Omega)$. These next two results, which make concrete the link between biased graphs and frame matroids, are standard (see [\ref{zas94}], for example). 

\begin{theorem}\label{framejoint}
  If $M$ is a matroid framed by a set $B$, then $M = \FM(\Omega)$ for some biased graph $\Omega$ in which $B$ comprises one unbalanced loop at each vertex. 
\end{theorem}

\begin{theorem}\label{framenojoint}
  A rank-$r$ matroid $M$ is frame if and only if there is some $r$-vertex biased graph $\Omega$, with no degree-zero vertices, for which $M = \FM(\Omega)$. 
\end{theorem}

Graphic matroids are a special case of frame matroids, and have a nice characterization as the unique frame matroids that are `rank-deficient'.

\begin{lemma}\label{framegraph}
  If $(G,\cB)$ is a connected biased graph on $n$ vertices, then either
  \begin{itemize}
      \item $r(\FM(G,\cB)) = n$, or
      \item $r(\FM(G,\cB)) = n-1$, and $(G,\cB)$ is balanced, and $\FM(G,\cB) = M(G)$.
  \end{itemize}
\end{lemma}
\begin{proof}
  Let $M = \FM(G,\cB)$, and suppose that $r(M) \ne n$. If $G$ has an unbalanced cycle $C$, then by connectedness of $G$ there is a set $X$ for which $C \subseteq X \subseteq E(G)$ and $|X| = n$, while $C$ is the unique cycle of $G$ contained in $X$. Then $X$ contains no circuit of $M$, so is independent in $M$; therefore $r(M) \ge |X| = n$. We also have $r(M) \le n$, since any set of more than $n$ edges of $G$ has a component containing least two cycles of $G$ and therefore a circuit of $M$. This contradicts $r(M) \ne n$. 
  
  Therefore every cycle of $G$ is balanced. Thus, $\FM(G,\cB)$ and $M(G)$ have the same circuits, so are the same matroid; this gives $r(M) = |V(G)|-1 = n-1$. 
\end{proof}

\begin{lemma}\label{frameloopgraphic}
  Let $(G,\cB)$ be a connected biased graph, let $M = \FM(G,\cB)$, and let $L$ be a set of unbalanced loops of $G$. If $r(M \del L) < r(M)$, then $M = M(G')$, where $G'$ is the graph obtained from $G$ by adding a new vertex $w$, and replacing each loop $e \in L$ with an edge $e$ from its vertex to $w$. 
\end{lemma}
\begin{proof}
  Since $G$ and $G \del L$ are connected, Lemma~\ref{framegraph} gives $|V(G)|-1 \le r(M \del L) < r(M) \le |V(G)|$, so $(G, \cB) \del L$ is balanced, and $M \del L = \FM((G, \cB) \del L) = M(G \del L)$. The circuits of $M$ are therefore precisely the cycles of $G \del L$, together with the sets of edges consisting of two loops in $L$ with a path between them. By the definition of $G'$, these correspond respectively to the cycles of $G'-w$, and the cycles of $G'$ containing $w$. Therefore $M(G')$ and $M$ have the same circuits, so $M = M(G')$. 
\end{proof}

We will be concerned with rank-$2$ uniform minors; in simple frame matroids, these are often easy to spot. We use the following lemma freely. Note that the span includes loops.

\begin{lemma}\label{frameline}
  Let $k \in \posi$ with $k \ge 2$, and $(G,\cB)$ be a biased graph for which $\FM(G,\cB)$ is simple. If $X$ is a set of two vertices of $G$ that spans $k$ edges of $G$, then $\FM(G,\cB)$ has a $U_{2,k}$-restriction. 
\end{lemma}
\begin{proof}
  By the definition of $\FM(G,\cB)$, any three edges spanned by $X$ form a dependent set of $\FM(G,\cB)$, so by simplicity are a triangle of $\FM(G,\cB)$. Therefore $\FM(G,\cB)|X \cong U_{2,|X|}$. Since $|X| \ge k$, this give the result.  
\end{proof} 

\begin{figure}
  \centering
  \parbox{5.5cm}
  { \includegraphics[width=5cm]{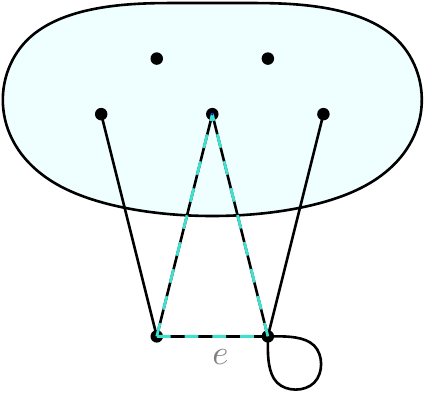}
    }
  \qquad
  \begin{minipage}{5.5cm}
  \includegraphics[width=5cm]{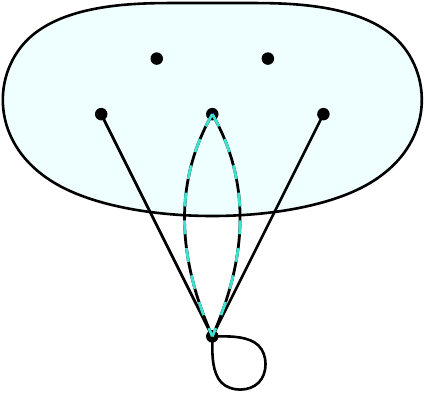}
  \end{minipage}
  \caption{The contraction of an edge $e$ in a biased graph. (Un)balanced cycles containing $e$ become (un)balanced cycles not containing $e$ }
\end{figure}
\label{contractfig}

Biased graphs have a minor order that corresponds to the minor order on their associated frame matroids, in a way very close to how graph minors correspond to graphic matroid minors. This is also described in Theorem 6.10.7 of [\ref{oxley}]; it works mostly as one would expect, except that contracting an unbalanced loop is a little awkward. Let $e$ be an edge of a biased graph $\Omega = (G,\cB)$. Define $\Omega \del e = (G \del e, \cB \del e)$, where $\cB \del e$ is the set of balanced cycles in $\cB$ that do not contain $e$. If $e$ is a nonloop edge of $G$, then define $\Omega \con e = (G \con e, \cB \con e)$, where $\cB \con e$ is the set of cycles $C$ of $G \con e$ for which either $C$ or $C \cup \{e\}$ is in $\cB$ (see Figure~\ref{contractfig}). If $e$ is a balanced loop edge of $G$, define $\Omega \con e = \Omega \del e$. If $e$ is an unbalanced loop at a vertex $v$, then define $\Omega \con e$ to be the biased graph obtained from $(G \del e, \cB)$ by replacing every other unbalanced loop at $v$ with a balanced loop at $v$, and every edge $f$ from $v$ to some $w \ne v$ with an unbalanced loop at $w$. (So $v$ is a vertex of $\Omega \con e$ incident only with balanced loops). This definition gives that $\FM(\Omega) \con e = \FM(\Omega \con e)$ and $\FM(\Omega \del e) = \FM(\Omega) \del e$ for each $e$. A \emph{minor} of $\Omega$ is a biased graph obtained from $\Omega$ by a sequence of vertex-deletions, edge-deletions, and contractions. 

\begin{lemma}\label{frameminor}
  Let $\Omega$ be a biased graph. If $N$ is a minor of the matroid $\FM(\Omega)$, then there is a minor $\Omega'$ of $\Omega$ for which $N = \FM(\Omega')$. 
\end{lemma}

By the way minors of biased graphs are defined, we also have the following. (The only nontrivial case to consider is contraction of an unbalanced loop, but such a contraction cannot create nonloop edges, so could be replaced with a vertex deletion without altering the presence of the minor $(G',\cB')$)

\begin{lemma}\label{looplessbiasedminor}
  If $(G', \cB')$ is a minor of the biased graph $(G,\cB)$, and $G'$ has no loops, then $G'$ is a minor of $G$. 
\end{lemma}

A rich family of examples of biased graphs comes from labellings over groups. Let $\Gamma$ be a group. A \emph{gain graph} is a pair $(G,\gamma)$, where $\gamma$ assigns to each edge of $G$ both a \emph{direction} (that is, an ordering of its two ends) and a \emph{gain} $\gamma(e) \in \Gamma$. We call a loop $e$ \emph{balanced} in $(G,\gamma)$ if and only if $\gamma(e) = 1_\Gamma$. A nonloop cycle $C = v_1, e_1, \dotsc, e_{k-1}, v_k = v_1$ of $G$ is \emph{balanced} in $(G, \gamma)$ if $\prod_{i=1}^{k-1} \gamma'(e_i) = 1_\Gamma$, where $\gamma'(e_i) = \gamma (e_i)$ if $e_i$ is directed from $v_i$ to $v_{i+1}$ by $\gamma$, and $\gamma'(e_i) = \gamma(e_i)^{-1}$ otherwise. It is routine to show that whether a cycle is balanced does not depend on the choice of the starting vertex $v_1$, nor in which direction the cycle is traversed, even if $\Gamma$ is nonabelian. 

Reversing the orientation of an edge and inverting its label also does not change the set of balanced cycles. Using this fact, one can prove that $(G,\cB_\gamma)$ is a biased graph, where $\cB_\gamma$ is the set of balanced cycles of $(G,\gamma)$. We call matroids of the form $\FM(G,\cB_\gamma)$ for a gain graph $(G,\gamma)$ the \emph{$\Gamma$-frame matroids}. They are also variously called \emph{gain-graphic} or \emph{voltage-graphic} over $\Gamma$.

If $\Gamma$ is finite, it is easy to check directly that $U_{2,|\Gamma| + 3}$ is not a $\Gamma$-frame matroid (the group is too small). 
The $\Gamma$-frame matroids form a minor-closed class [\ref{zas99}]; the next lemma is a consequence of these two facts. 

\begin{lemma}
  If $\Gamma$ is a finite group, then no $\Gamma$-frame matroid has a $U_{2,|\Gamma|+3}$-minor. 
\end{lemma}

We can use gains to `blow up' a graph $G$ using a finite group $\Gamma$ into a biased graph $G^{\Gamma}$. Given a simple graph $G$ and an integer $t$, let $G^t$ be the graph obtained from $G$ by adding a loop at each vertex, and replacing each edge with a parallel class of size $t$. If $\Gamma$ is trivial, let $G^{\Gamma}$ denote the biased graph $(G^1, \cB)$ in which the unbalanced cycles are precisely the loops. Otherwise, let $(G^{|\Gamma|},\gamma)$ be a $\Gamma$-gain graph in which $\gamma(e) \ne 1_\Gamma$ for each loop $e$, and each nonloop edge $e$ of $G$ corresponds to a set of $|\Gamma|$ edges directed the same way, being assigned the $|\Gamma|$ distinct elements of $\Gamma$ as gains. Let $G^{\Gamma}$ denote the biased graph $(G^{|\Gamma|},\cB_\gamma)$ arising from this labelling. By construction, the matroid $\FM(G^{\Gamma})$ has rank equal to $|V(G)|$, and has exactly $|V(G)| + |\Gamma| |E(G)|$ elements (this also holds when $\Gamma$ is trivial). This frame matroid is also independent of the particular choice of orientations and gains of each edge, since every possible gain appears between every pair of vertices. 

\begin{figure}
  \includegraphics[width=5cm]{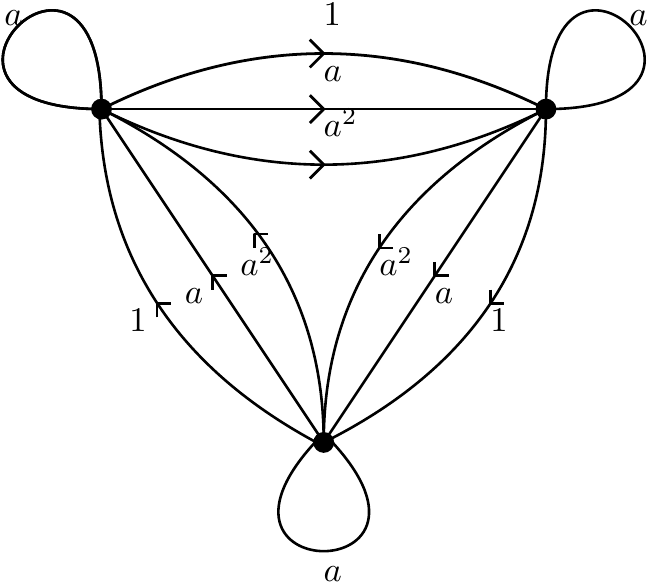}
  \caption{A labelled graph over a three-element group $\{1,a,a^2\}$ that would give rise to a rank-$3$ Dowling Geometry. }
\end{figure}

\label{dgfigure}

In the special case where $G$ is a complete graph with vertex set $[n]$, the matroid $\FM(G^{\Gamma})$ is called a rank-$n$ \emph{Dowling Geometry} over $\Gamma$, and is written $\DG(n,\Gamma)$. See Figure~\ref{dgfigure}. These matroids play roles analogous to those played by complete graphs and projective geometries, but for the $\Gamma$-frame matroids. We have $|\DG(n,\Gamma)| = n + |\Gamma|\binom{n}{2}$. 
We recognize Dowling Geometries via the follow theorem, which was essentially proved in~[\ref{kk}]. (See Theorem 5.1 of [\ref{gnw}] for an explicit proof.)

\begin{theorem}\label{dowlingrec}
  Let $t,n \in \posi$ with $n \ge 4$, and $\Omega$ be an $n$-vertex biased graph with  one loop at each vertex, and $t$ edges between every pair of vertices. If $\FM(\Omega)$ is simple and in $\cU(t+1)$, then $\FM(\Omega) \cong \DG(n,\Gamma)$ for some group $\Gamma$ of order $t$. 
\end{theorem}

In general, the question of when frame matroids are graphic can be subtle [\ref{cdfp}]. However, the question of which frame matroids are the cycle matroids of cliques has a simple answer. This proof is standard, but we include it for completeness. 

\begin{lemma}\label{cliqueframerep}
  Let $t \in \bN$ with $t \ge 3$, and $\Omega$ be a biased graph with no degree-zero vertices. Then $\FM(\Omega) \cong M(K_{t+1})$ if and only if $\Omega$ is either a balanced $K_{t+1}$, or a balanced $K_t$ with an unbalanced loop added at every vertex.
\end{lemma}
\begin{proof}
  We start with the backwards direction. If $\Omega$ is a balanced $K_{t+1}$, then the circuits of $\FM(\Omega)$ coincide exactly with the cycles of $M(K_{t+1})$, so $\FM(\Omega) = M(K_{t+1})$. If $\Omega = (G,\cB)$ is a balanced $K_{t}$ with an unbalanced loop at each vertex, then let $L$ be its set of unbalanced loops; using Lemma~\ref{framegraph} we have $r(\FM(\Omega) \del L) = r(M(K_t)) = t -1 < t = r(\FM(\Omega))$, so Lemma~\ref{frameloopgraphic} implies that $\FM(\Omega) = M(G')$, where $G' \cong K_{t+1}$. 
  
  For the forwards direction, let $\Omega = (G,\cB)$ with $M = \FM(\Omega)$, and $G = (V,E,\iota)$. If $G$ is disconnected, then $M \cong M(K_{t+1})$ is disconnected, a contradiction. So $G$ is connected. If every cycle of $G$ is balanced, then since $M$ is simple, so is $G$; we thus have $M(K_{t+1}) \cong \FM(G, \cB) = M(G)$, and so clearly $G \cong K_{t+1}$ as required. Otherwise, there is some unbalanced cycle, so  $|V| = r(M) = t$ by Lemma~\ref{framegraph}.
  
  Let $L$ be the set of loops of $G$. For each vertex $v$, let $E_v$ be the set of edges of $G$ incident with $v$. The matroid $M \del E_v$ has rank at most $t-1$; since $M \cong M(K_{t+1})$ we have $|E_v| \ge t$ for all $v$. Therefore 
  $t^2 \le \sum_{v \in V(G)} |E_v| = 2|E| - |L| = t^2 + t - |L|$, so $|L| \ge t$.
  Since $M$ is simple, each loop is unbalanced and there is at most one at each vertex; therefore $|L| = t$. 

  The set $L$ is thus a basis of $M = \FM(\Omega)$, but also $M = M(K_{t+1})$ is binary, so there is a $\GF(2)$-representation $A$ of $M$ for which $A[L]$ is an identity matrix. Each nonloop edge $e$ is in a triangle with the loops at its ends, so the vector $A_e$ has support $2$, and the columns of $A[C]$ sum to zero for each nonloop cycle $C$ of $G$. It follows that $G-L \cong K_t$, and every nonloop cycle $C$ of $G$ is dependent in $M$, and therefore a balanced cycle of $G$.
\end{proof}

Given a set $B$, a matroid $M$ is a \emph{$B$-clique} if $M$ is framed by $B$, while each pair of elements of $B$ is contained in a triangle of $M$. We show that these objects often correspond to graphic cliques. (An \emph{edge-neighbourhood} is the set of edges incident with some vertex.)

\begin{lemma}\label{cliquetographic}
  If $M$ is a $B$-clique and $M \del B$ is simple while $r(M \del B) < r(M)$, then $M$ is the cycle matroid of a complete graph in which $B$ is an edge-neighbourhood.
\end{lemma}
\begin{proof}
  By Theorem~\ref{framejoint}, we have $M = \FM(G,\cB)$ for some biased graph $(G,\cB)$ in which $B$ comprises an unbalanced loop at each vertex; let $n = |V(G)|$. Since $M$ is a $B$-clique, the graph $G$ has at least one edge between every pair of vertices, so is connected. Lemma~\ref{frameloopgraphic} therefore gives $M = M(G')$, where $G'$ is obtained from $G$ by adding a new vertex $w$ and an edge $e$ from $w$ to the vertex of $e$ in $G$, for each $e \in B$. Since $M(G) = M \del B$ is simple with an edge between every pair of vertices, we have $G \cong K_{|B|}$, so $G' \cong K_{|B|+1}$, with $B$ as an edge-neighbourhood.
\end{proof}

\section{Lower bounds}\label{crownsection}

We start by constructing the `crowns' that provide the lower bound in Theorem~\ref{main2}. For each prime power $q$, and $t,n \in \nni$ with $t \le n$, we define a certain rank-$n$, $\GF(q)$-representable matroid. Let $B$ be a basis for $G \cong PG(n-1,q)$ and $B_0$ be a $t$-element subset of $B$. Let $X = \cup_{e \in B - B_0} \cl_M(B_0 \cup \{e\})$; an \emph{$(n,q,t)$-crown} is a matroid isomorphic to $G|X$. Crowns are clearly determined uniquely by the choice of $n,q$ and $t$, and are simple and $\GF(q)$-representable. An $(n,q,0)$-crown is a free matroid, and an $(n,q,n)$-crown is isomorphic to $\PG(n-1,q)$. Let $K_t^-$ be the graph obtained from $K_t$ by removing an edge. 

\begin{lemma}\label{crown}
    Let $t,n \in \posi$ with $t < n-1$, and $q$ be a prime power. If $M$ is an $(n,q,t)$-crown, then 
    \begin{enumerate}[(i)]
        \item\label{cri} $|M| = (n-t) q^t + \frac{q^t-1}{q-1}$, 
        \item\label{crii} $M$ has an $M(K_{t+3}^-)$-restriction, but no $M(K_{t+3})$-minor. 
    \end{enumerate}
\end{lemma}
\begin{proof}
    Define $B,B_0$ in $G$ as in the definition. Let $F_0 = \cl_G(B_0)$, and $F_e = \cl_G(B_0 \cup \{e\})$ for each $e$. Note that $|F_0| = \frac{q^t-1}{q-1}$ and $|F_e - F_0| = \frac{q^{t+1} - 1}{q-1} - \frac{q^t-1}{q-1} = q^t$. Since $B$ is a basis, if $e,e' \in B - B_0$ are distinct, then $F_e \cap F_{e'} = F_0$; therefore
    \[ |X| = \sum_{e \in B - B_0}|F_e - F_0| + |F_0| = (n-t)q^t + \frac{q^t-1}{q-1}, \]
    giving (\ref{cri}). 
    
    Since $t < n-1$, we have $|B-B_0| \ge 2$; let $e,f$ be distinct elements of $B-B_0$. Consider an $M(K_{t+3})$-restriction $K$ of $G$ having $B' = B_0 \cup \{e,f\}$ as a vertex-neighbourhood. Since every element of $K$ is on a different triangle containing two elements of $B'$, and all such triangles except the one containing $e$ and $f$ are contained in $F_e \cup F_f$, it follows that $K \del x$ is a restriction of $M$ for some $x \in E(K)$. Since $K \del x \cong M(K_{t+3}^-)$, it follows that $M$ has such a restriction. 

    Suppose for a contradiction that $M$ has an $M(K_{t+3})$-minor; let $C$ be a set so that $M \con C$ contains $M(K_{t+3})$ as a spanning restriction. Since the class of matroids with a spanning clique restriction is closed under contraction, the matroid $N = M \con (C \cup B_0)$ also has a spanning clique restriction, while $r(N) \ge r(M \con C) - |B_0| = 2$. Therefore $N$ has a triangle. However, $N$ is a minor of the matroid $M \con B_0$, in which every nonloop element is parallel to a member of the independent set $B - B_0$; this is a contradiction. 
\end{proof}

Lemma~\ref{crown} provides the lower bounds in Theorem~\ref{main2}. We now construct matroids that are denser than crowns for certain choices of parameters. For this, we need a lemma. This can also be proved by considering critical numbers (see [\ref{oxley}], Section 15.3), but we give an elementary argument. 

\begin{lemma}\label{agclique}
  If $q$ is a prime power, then $\AG(q-1,q)$ has no $M(K_{q+1})$-minor. 
\end{lemma}
\begin{proof}
  Let $K$ be an $M(K_{q+1})$-minor of $G \cong \AG(q-1,q)$. Since $r(M) = q = r(G)$, it follows that $K$ is a restriction of $G$. Let $D$ be the $\GF(q)$-representation of $M(K_{q+1})$ with rows indexed by $[q+1]$, containing as columns the $\binom{q+1}{2}$ vectors $e_i - e_j : i < j$, where $e_1, \dotsc, e_{q+1}$ is the standard basis of $\GF(q)^{q+1}$. The matroid $K$ is binary, so is uniquely representable over every field (See [\ref{oxley}], Proposition 6.6.5). Therefore $G$ has a $\GF(q)$-representation $A$ with $A[E(K)] = D$. 
  
  Since $G$ is an affine geometry, there is a hyperplane $H$ of $\GF(q)^{q+1}$ such that $A$ has no column in $H$, so there is a vector $w \in \GF(q)^{q+1}$ such that $x^{T}w \ne 0$ for each column $x$ of $A$.
   Since $q + 1 > q$, there exist $i < j$ with $w_i = w_j$. Therefore $e_i-e_j$ is a column of $A$ with $(e_i-e_j)^Tw = 0$, a contradiction.
\end{proof}

If $q$ is a prime power, the densest rank-$n$ crowns without a $U_{2,q+2}$-minor or $M(K_{q+1})$-minor have $(n-q+2)q^{q-2} + \frac{q^{q-2}-1}{q-1} \approx nq^{q-2}$ elements. The next lemma shows that there are denser examples, having around $\frac{q^{q-1}-1}{q-1}n$ elements. This is the reason for the stipulation that $t$ is large in Conjecture~\ref{mainconj}.  

\begin{lemma}\label{coupled}
  Let $q$ be a prime power, and let $n \in \posi$ with $n \equiv 1 \pmod{q-1}$. There exists a simple rank-$n$ matroid $M \in \cU(q)$ with no $M(K_{q+1})$-minor, satisfying $|M| = (n-1)\frac{q^{q-1}-1}{q-1} + 1$. 
\end{lemma}
\begin{proof}
  When $q = 2$, we can choose $M$ to be a free matroid, so assume that $q \ge 3$. Let $t = \tfrac{n-1}{q-1} \in \posi$, and let $W \cong \AG(q-1,q)$ be the rank-$q$ affine geometry over $\GF(q)$. Let $M$ be the parallel connection of $t$ copies of $W$ at a common point. Then $r(M) = t(r(W)-1) + 1 = n$, and $|M| = t(|W|-1) + 1 = (n-1)\tfrac{q^{q-1}-1}{q-1} + 1$. Since $M$ is $\GF(q)$-representable, we have $M \in \cU(q)$, so it suffices to show that $M$ has no $M(K_{q+1})$-minor; suppose that $M$ has an $M(K_{q+1})$-minor. 
  
  Since $M$ is the parallel connection of copies of $W$ and $M(K_{q+1})$ is $3$-connected, it follows that $W$ has an $M(K_{q+1})$-minor $K$. This contradicts Lemma~\ref{agclique}. 
\end{proof}

\section{Frame Matroids}\label{framesection}

The goal of this section is to prove Theorems~\ref{mainframeclique} and~\ref{mainframe}.
We start with Theorem~\ref{mainframe}, proving it in a stronger form which identifies the extremal examples. The stipulation that $r \ne 3$ is necessary; maximal rank-$3$ frame matroids in $\cU(\ell)$ that are not Dowling geometries can be constructed from quasigroups [\ref{kk}]. 

\begin{theorem}\label{extremalframe}
  Let $\ell,r \in \posi$. If $M$ is a simple rank-$n$ frame matroid in $\cU(\ell)$, then $|M| \le r + (\ell-1)\binom{r}{2}$. If $\ell \ne 1, r \ne 3$ and equality holds, then $M \cong \DG(r,\Gamma)$ for some group $\Gamma$ of order $\ell-1$.
\end{theorem}
\begin{proof}
  If $\ell = 1$, the result is easy. If $\ell = 2$, then $M$ is binary; by Corollary 4.2 of [\ref{zas1}], it is regular. Heller [\ref{heller}] proved that simple rank-$r$ regular matroids have at most $\binom{r+1}{2}$ elements with equality precisely for $M(K_{r+1}) \cong \DG(r,\{1\})$, giving our result. Suppose, therefore, that $\ell \ge 3$. 

  Let $f(x) = x + (\ell-1)\binom{x}{2}$, noting that $f(x) = f(x-1) + (\ell-1)(x-1) + 1$.

  Let $r$ be minimal so that the theorem fails for some $M$, so $|M| \ge f(r)$, and either $|M| > f(r)$, or $|M| = f(r)$ and $r \ne 3$ while $M$ is not one of the stated examples. Since $f(1) = 1$ and $f(2) = \ell+1$ while $\DG(1,\Gamma) \cong U_{1,1}$ and $\DG(2,\Gamma) \cong U_{2,|\Gamma|+2}$ are the unique maximal matroids of their ranks in $\cU(\ell)$ whenever $\Gamma$ is the cyclic group of order $\ell-1$, we have $r \notin \{1,2\}$, so $r \ge 3$. 

  Let $(G,\cB)$ be an $r$-vertex biased graph with $M = \FM(G,\cB)$, and let $V = V(G)$. For all distinct $x,y \in V$, let $n_{xy}$ be the number of edges of $G$ from $x$ to $y$, and $n_x = n_{xx} \in \{0,1\}$ be the number of loops at $x$. For distinct $x,y \in V$, let $d_{xy} = \ell+1 - (n_{xy} + n_x + n_y)$; since $M \in \cU(\ell)$ we have $d_{xy} \ge 0$. Let $b_3 = 1$ if $r = 3$, and $b_3 = 0$ otherwise; by assumption we have $\elem(M) \ge f(r) + b_3$. 
 
  \begin{claim}\label{trimin}
    Let $x,y \in V$ with $n_{xy} \ge 2$. For each edge $e$ from $x$ to $y$, the number of balanced triangles of $G$ containing $e$ is at least $(\ell-1)(r-2) + d_{xy} + b_3$.
  \end{claim}
  \begin{subproof}
    Let $\cT$ be the set of balanced triangles containing $e$. When $e$ is contracted, the $n_x + n_{xy} + n_y - 1 \ge 1$ other edges spanned by $\{x,y\}$ become parallel, as do the two edges in $T-\{e\}$ for each $T \in \cT$. Therefore
    \[\elem(M) - \elem(M \con e) = (n_x + n_{xy} + n_y - 1) + |\cT| = \ell + d_{xy} - |\cT|.\] 
    By the choice of $M$ and the fact that $M \con e$ is not a counterexample, we also have $\elem(M) - \elem(M \con e) \ge (f(r) + b_3) - f(r-1) = (\ell-1)(r-1) + 1 + b_3$; combining these estimates gives the claim. 
  \end{subproof}

  \begin{claim}\label{allgood}
    $n_{xy} \le \ell-1$ for all $x,y \in V(G)$. 
  \end{claim}
  \begin{subproof}
    Suppose not; choose a pair $x,y$ with $n_{xy}$ as large as possible (so $n_{xy} \ge \ell$), and $n_x \ge n_y$. Let $e$ be an edge from $x$ to $y$. For each vertex $z$ of $G-\{x,y\}$, let $\cT_z$ be the set of balanced triangles containing $z$ and $e$. Since no two triangles in $\cT_z$ have an edge from $x$ to $z$ or an edge from $y$ to $z$ in common, the choice of $x$ and $y$ gives $n_{xy} \ge \min(n_{xz},n_{yz}) \ge |\cT_z|$ for all $z$. 
    
    By a majority argument using~\ref{trimin} and $d_{xy},b_3 \ge 0$, there exists $w$ for which $|\cT_w| \ge \ell-1 \ge 2$. Then $n_{yw} \ge |\cT_w| \ge 2$, so there are distinct edges $e_1,e_2$ from $y$ to $w$. If $n_x + n_{xy} \ge \ell+1$, then the minor $(G,\cB) \con e_2$ has an unbalanced loop $e_1$ at $y$, as well as $n_x + n_{xy} \ge \ell+1$ edges from $x$ to either $x$ or $y$; this contradicts Lemma~\ref{frameline}. So $n_y + n_{xy} \le n_x + n_{xy} \le \ell \le n_{xy}$, giving $n_{xy} = \ell$ and $n_x = n_y = 0$. Thus $d_{xy} = 1$, and $|\cT_w| \le n_{xy} = \ell$. 
    
    If there is a vertex $w' \notin \{x,y,w\}$ for which $|\cT_{w'}| \ge 2$, then $n_{xw'} \ge 2$, so there are distinct edges $e_1',e_2'$ from $w'$ to $x$. Now the minor $(G, \cB) \con \{e_2,e_2'\}$ has unbalanced loops $e_1',e_1$ at $x,y$ respectively, and $n_{xy} = \ell$ edges from $x$ to $y$; this contradicts Lemma~\ref{frameline}. So $|\cT_{w'}| \le 1$ for all $w' \in V - \{x,y,w\}$. 
    Using $d_{xy} = 1$ and~\ref{trimin} gives
    \[(\ell-1)(r-2) + 1 + b_3 \le \sum_{z \in V - \{x,y\}}|\cT_{z}| \le |\cT_w| + |V - \{x,y,z\}| \le \ell + r-3,\]
    This implies that $(\ell-2)(r-3) + b_3 \le 0$, which contradicts $\ell,r \ge 3$. 
  \end{subproof}

  The fact that $|M| \le f(r)$ now follows immediately, because $G$ has at most $r$ loop edges, and~\ref{allgood} implies that $G$ has at most $(\ell-1)\binom{r}{2}$ nonloop edges. If $|M| = f(r)$, then $G$ has $r$ loops and every pair of vertices spans $(\ell-1)$ edges. Now by Theorem~\ref{dowlingrec} with $t = \ell-1$, it follows that $G \cong \DG(r,\Gamma)$ for some group $\Gamma$ of order $\ell-1$, contrary to the choice of $M$ as a counterexample.
\end{proof}

This next lemma is the tool we use to reduce Theorem~\ref{mainframeclique} to a problem in graph theory; it shows that a dense frame matroid in $\cU(\ell)$ has a dense graphic minor, where a factor of roughly $2(\ell-1)$ is lost in the density. 

\begin{lemma}\label{denseframetographic}
  Let $\ell \in \posi$ and $c \in \bR$. If $M$ is a simple frame matroid for which $M \in \cU(\ell)$ and $|M| > (2c(\ell-1) - \ell + 2)r(M)$, then $M$ has a simple graphic minor $N$ with $|N| > c r(N)$. 
\end{lemma}
\begin{proof} 
  Suppose for a contradiction that no such $N$ exists.
  The hypothesis implies that $M$ is nonempty, so $c \ge 1$, since otherwise $N \cong U_{1,1}$ contradicts this assumption. It follows that $|M| > r(M)$, so $M$ has a $U_{2,3}$-minor, which provides a contradiction for $c < \frac{3}{2}$; thus $c \ge \tfrac{3}{2}$. 

  Let $c_1 = 2c(\ell-1) - \ell + 2$. Let $M'$ be a minimal minor of $M$ such that $\elem(M') > c_1 r(M')$. Let $r = r(M')$, and note that $M'$ is simple. Let $(G,\cB)$ be a biased graph on $r$ vertices, for which $M' = \FM(G,\cB)$, where $G = (V,E(M'))$.
  
  \begin{claim}\label{connsix}
     $G$ is connected, and has at least four vertices. 
  \end{claim}
  \begin{subproof}
    If $G$ is disconnected, then so is $M'$, then $M' = M_1 \oplus M_2$ for nonempty simple matroids $M_1,M_2$ with $r(M_1) + r(M_2) = r$. By the minimality of $M'$, we have $|M_i| \le c_1 r(M_i)$ for each $i$, and so $|M'| = |M_1| + |M_2| \le c_1(r(M_1) + r(M_2)) = c_1 r(M')$, a contradiction. 

    Since $M'$ is a simple frame matroid in $\cU(\ell)$, Theorem~\ref{extremalframe} and the choice of $M'$ give that $c_1 r < |M'| \le r + (\ell-1) \tbinom{r}{2}$. This gives $(c_1-1)r < (\ell-1)\tbinom{r}{2}$, which simplifies to $r > 4c-1$. Since $|V(G)| = r$ and $c \ge 1$, the claim follows. 
  \end{subproof}

  Let $S_v$ denote the set of neighbours of a vertex $v$ in $G$, other than $v$ itself. For $u,v \in V$, let $n_{uv}$ denote the number of edges of $G$ from $u$ to $v$, and let $n_u = n_{uu}$ be the number of loops at a vertex $u$. For distinct $x,y,z \in V$, let $\cT_{xy}$ be the set of balanced triangles of $G$ containing $x$ and $y$, and $\cT_{xyz}$ be the set of balanced triangles containing $x,y$ and $z$. 
  
  \begin{claim}\label{tribound}
    If $x,y \in V$ are adjacent, then $|\cT_{xy}| > (c_1 - \max(1,n_x + n_{xy} + n_y-1)) n_{xy}$.
  \end{claim}
  \begin{subproof}
    Let $n'_{x,y} = n_x + n_{xy} + n_y$. For each edge $e$ from $x$ to $y$, let $\cT_e$ be the set of balanced triangles of $G$ containing $e$. When $e$ is contracted from $M'$, the $n'_{xy}-1$ edges with ends in $\{x,y\}$ other than $e$ become a parallel class, as do the two other edges of each triangle in $\cT_e$. It follows that 
    \[\elem(M') - \elem(M' \con e) = |\cT_e| + 1 + \max(0,n'_{xy} -2) = |\cT_e| + \max(1,n'_{xy}-1).\]
    Using the minimality in the choice of $M'$, we have $\elem(M') - \elem(M' \con e) > c_1 r(M') - c_1 r(M' \con e) = c_1$, so  $|\cT_e| + \max(1,n'_{xy}-1) > c_1$. Since $\cT_{xy}$ is the disjoint union of the $\cT_e$ over the edges $e$ from $x$ to $y$, summing this over each of the $n_{xy}$ edges $e$ from $x$ to $y$ gives the claim. 
  \end{subproof}

  \begin{claim}\label{tribound1}
    If $x,y \in V$ are adjacent, then $\frac{|\cT_{xy}|}{n_{xy}} >  2(\ell-1)(c-1)$. 
  \end{claim}
  \begin{subproof}
    Since $M \in \cU(\ell)$, we have $\max(1,n_x + n_{xy} + n_y-1) \le \ell$. Thus, by~\ref{tribound}, we have $\frac{|\cT_{xy}|}{n_{xy}} > (2c(\ell-1) - \ell + 2) - \ell = 2(\ell-1)(c-1)$.
  \end{subproof}

  For each $x \in V$ and $y \in S_x$, let $\beta(x,y) = \frac{1}{2n_{xy}}\sum_{z \in S_x - \{y\}} \frac{|\cT_{xyz}|}{n_{xz}}$. This is some weighted count of the number of triangles containing $x$ and $y$; our next claim uses the assumption that no $N$ exists to show that $\beta$ is not always large. 
  
  \begin{claim}\label{betasmall}
    For all $x \in V$, there is some $y \in S_x$ for which $\beta(x,y) \le c-1$.
  \end{claim}
  \begin{subproof}
    Suppose that $x$ is such that $\beta(x,y) > c - 1$ for all $y \in S_x$. Let $s = |S_x|$. Note that $s > 0$ by~\ref{connsix}. 
    
    Let $J \subseteq E(G)$ be obtained by choosing one of the edges from $x$ to $y$ uniformly at random for each $y \in S_x$, and let $F = \cl_{M'}(J)$. Since $J$ is a star at $x$, it spans no other edge incident with $x$, so every edge in $F-J$ has both ends in $S_x$. Then $M' | F$ is a frame matroid represented by the $s$-vertex biased subgraph of $(G,\cB)[\{x\} \cup S_x]$ with edge set $F$; in this graph, the vertex $x$ has exactly one edge to each other vertex. Therefore $M'|F$ is connected and has rank $|J| = s$, so is graphic by Lemma~\ref{framegraph}.

    Let $y,z \in S_x$ be distinct, and $T \in \cT_{xyz}$. The probability that $E(T) \subseteq F$ is the probability that $J$ contains both the edges of $T$ incident with $x$, which is equal to $\tfrac{1}{n_{xy}n_{xz}}$. By linearity, the expected number of $T \in \cT_{xyz}$ with $E(T) \subseteq F$ is $\frac{|\cT_{xyz}|}{n_{xy}n_{xz}}$. Summing over all pairs $y,z$, we see that the expected number of $T \in \cT_x$ with $E(T) \subseteq F$ is therefore 
    \begin{align*}
      \sum_{\{y,z\} \in \binom{S_x}{2}} \frac{|\cT_{xyz}|}{n_{xy}n_{xz}} = \sum_{y \in S_x} \frac{1}{2n_{xy}} \sum_{z \in S_x-\{y\}} \frac{|\cT_{xyz}|}{n_{xz}} = \sum_{y \in S_x} \beta(x,y) > (c-1) s
    \end{align*}
    Therefore there is some choice $J_0$ for $J$ such that $F_0 = \cl_{M'}(J_0)$ contains the edge sets of more than $(c-1) s$ triangles in $\cT_{x}$. Let $\cT_0$ be the set of triangles $T \in \cT_x$ with $E(T) \subseteq F_0$. If $T,T'$ are distinct triangles in $\cT_0$ for which $V(T) = V(T')$, then since $x$ has exactly one edge in $F_0$ to each vertex in $S$, we have $E(T) \cap J = E(T') \cap J$, which implies that $T$ and $T'$ have two edges in common, so are equal. Therefore the triangles in $\cT_0$ have distinct vertex sets, but all contain $x$. In particular, no two triangles in $\cT_0$ contain the same edge in $F_0-J_0$; it follows that $|F_0-J_0| \ge |\cT_0|$, and so $|F_0| \ge |\cT_0| + |J_0| \ge (c-1)s + s = cs$.
    Therefore $M$ has a simple rank-$s$ graphic minor $N = M'|F_0$ with $|N| > c r(N)$, a contradiction.
  \end{subproof}
  Our next claim asserts that every vertex is incident with a large parallel class.
  \begin{claim}\label{ldeg1}
    For all $x \in V$, there is some $y \in V$ for which $n_{xy} \ge \ell$. 
  \end{claim}
  \begin{subproof}
    If not, let $a \in V$ with $n_{ay} \le \ell-1$ for all $y$. For each $y \in S_a$, we have
    \begin{align*}
      \beta(a,y) = \frac{1}{2n_{ay}} \sum_{z \in S_a - \{y\}} \frac{|\cT_{ayz}|}{n_{az}} \ge \frac{1}{2(\ell-1)n_{ay}} \sum_{z \in S_a - \{y\}} |\cT_{ayz}| 
      = \frac{|\cT_{ay}|}{2(\ell-1)n_{ay}}.
    \end{align*}
    But using~\ref{tribound1}, this gives $\beta(a,y) > c-1$ for all $y \in S_a$, which contradicts~\ref{betasmall}.
  \end{subproof}
  
  \begin{claim}\label{baddeg2}
    For all $x \in V$, there are at least two $y \in V$ for which $n_{xy} \ge 2$. 
  \end{claim}
  \begin{subproof}
    Suppose that there is some $a \in V$ for which this fails. Let $b$ be the vertex given by~\ref{ldeg1} for which $n_{ab} \ge \ell$, and $y \in S_a$ be the vertex given by~\ref{betasmall} with $\beta(a,y) < c$. By assumption, we have $n_{ax} = 1$ for all $x \in S_a - \{b\}$. If $y = b$, then 
    \[\beta(a,y) = \frac{1}{2n_{ab}}\sum_{z \in S_a - \{b\}}\frac{|\cT_{abz}|}{n_{az}} = \frac{1}{2n_{ab}}\sum_{z \in \cS_a - \{b\}}|\cT_{abz}| = \frac{|\cT_{ab}|}{2n_{ab}} > (\ell-1)(c-1) \ge c-1,\]
    where the penultimate inequality uses~\ref{tribound1}; this is a contradiction. If $y \ne b$, then we have $|\cT_{aby}| \le n_{ab}n_{ay} = n_{ab}$ and $|\cT_{aby}| \le n_{by}n_{ay} = n_{by}$, so $|\cT_{aby}| \le \min(n_{ab},n_{by})$. 
    \begin{align*}
      \beta(a,y) &= \frac{1}{2n_{ay}}\sum_{z \in S_a - \{y\}}\frac{|\cT_{ayz}|}{n_{az}} \\ 
                 &= \frac{1}{2} \left(\left(\sum_{z \in S_a - \{y\}} |\cT_{ayz}|\right) + \left(\frac{1}{n_{ab}}-1\right)|\cT_{aby}|\right) \\
                 &= \tfrac{1}{2}\left(|\cT_{ay}| + \left(n_{ab}^{-1}-1\right)|\cT_{aby}|\right) \\
                 &> \tfrac{1}{2}\left(c_1 - \max(1,n_a + n_y) + \left(n_{ab}^{-1}-1\right)\min(n_{ab},n_{by})\right), \\
                 &= c(\ell-1) + 1 -  \tfrac{\ell}{2} - \tfrac{1}{2}\left(\max(1,n_a+n_y) + \left(1- n_{ab}^{-1}\right)\min(n_{ab},n_{by})\right);
    \end{align*}
    the inequality uses uses~\ref{tribound} and the fact that $n_{ay} = 1$, and the last line uses the definition of $c_1$. 
    
    Suppose that $n_y = 1$. If $n_{by}  > 0$, then since $n_{ay} = 1$, contracting the loop at $y$ creates unbalanced loops at both $a$ and $b$ and maintains the edges from $a$ to $b$; thus $M$ has a $U_{2,n_{ab}+2}$-minor; since $n_{ab} \ge \ell$, this contradicts Lemma~\ref{frameline}. So $n_{by} = 0$.  Also, $\max(1,n_a+n_y) = n_a + 1 \le 2$. The above lower bound on $\beta(a,y)$ thus gives 
    \[\beta(a,y) > c(\ell-1) + 1 - \tfrac{\ell}{2} - \tfrac{1}{2}(2 + 0) = (c - \tfrac{1}{2})(\ell-1) - \tfrac{1}{2} \ge c-1,\]
    a contradiction. 
    
    Therefore $n_y = 0$, so $\max(1,n_a+n_y) = 1$. This gives the inequality $c-1 > c(\ell-1) + \tfrac{1-\ell}{2} - \tfrac{1}{2}(1 - \tfrac{1}{n_{ab}})\min(n_{ab},n_{by})$, since $c \ge \tfrac{3}{2}$, this gives
    \[(1 - n_{ab}^{-1})\min(n_{ab},n_{by}) > 1 + (2c-1)(\ell-2) \ge 2\ell -3.\]
    Using $\min(n_{ab},n_{by}) \le n_{ab}$, the left-hand side is at most $n_{ab}-1 \le (\ell+1)-1$, so $\ell > 2\ell-3$, and thus $\ell = 2$. This in turn implies that $n_{ab}-1 > 1$, so $n_{ab} \ge 3$. If $n_{by} \le 1$, then we have the contradiction $(1 - n_{ab}^{-1}) > 1$, so $n_{by} \ge 2$. Contracting an edge between $b$ and $y$ therefore creates a loop at $b$ while maintaining the edges from $a$ to $b$; therefore $M$ has a $U_{2,n_{ab}+1}$-minor, since $n_{ab} \ge 3$ and $M \in \cU(2)$, this contradicts Lemma~\ref{frameline}.
  \end{subproof}

  Our final claim finds a path containing three large parallel classes that we will use to find a $U_{2,\ell+2}$-minor. 

  \begin{claim}
    $G$ has a path $P$ whose first and last edges are in parallel classes of $G$ of size at least $2$, and in which some other edge is in a parallel class of $G$ of size at least $\ell$. 
  \end{claim}  
  \begin{subproof}
    Let $x \in V$, and let $y \in V$ be given by~\ref{ldeg1} so that $n_{xy} \ge \ell$. By~\ref{baddeg2}, there are neighbours $w,w'$ of $x,y$ respectively so that $n_{xw} \ge 2$ and $n_{yw'} \ge 2$; choose $w \ne w'$ if possible. If $w \ne w'$, then $w,x,y,w'$ are the vertices of a path satisfying the claim, so we may assume that $w = w'$. If $\max(n_{xw},n_{yw}) < \ell$, then by~\ref{ldeg1} there is a vertex $z \notin \{x,y,w\}$ for which $n_{wx} \ge \ell$, and by~\ref{baddeg2} there is a vertex $z' \ne w$ for which $n_{zz'} \ge 2$. Now $z',z,w,v$ are vertices of a path as required, where $v$ is a vertex in $\{x,y\} - \{z'\}$. 

    Therefore $\max(n_{xw},n_{yw}) \ge \ell$; let $T$ be a triangle on $\{x,y,w\}$. Recall that $G$ is connected with at least four vertices; by~\ref{ldeg1}, there are thus vertices $z,z'$ of $G$ with $z \notin \{x,y,w\}$, while $n_{zz'} \ge \ell \ge 2$, and there is a path $P$ from $z'$ to $T$ whose internal vertices are not in $\{z,z'\} \cup V(T)$. Since $T$ contains two pairs $x,x'$ for which $n_{xx'} \ge \ell$, it is now clear that there is a path as required, beginning with $z,z'$ and ending with two vertices in $T$. 
  \end{subproof}

  Let $f_1$ and $f_2$ be edges outside $P$ that are parallel to the first and last edges of $P$, and let $F$ be the parallel class of size at least $\ell$ containing another edge of $P$. Let $J = E(P) - F$. It is easy to verify, using the definition of $\FM(G,\cB)$ and the fact that all $2$-cycles of $G$ are unbalanced, that for every $T \subseteq F \cup \{f_1,f_2\}$ with $|T| = 3$, the set $T \cup J$ contains a unique circuit $C_T$ of $M'$, and that $T \subseteq C_T$. It follows that $T$ is a triangle of $M' \con J$. This holds for every $T$, so the matroid $(M' \con J)|(F \cup \{f_1,f_2\})$ is uniform of rank $2$. Since $|F| \ge \ell$, this a a contradiction.
\end{proof}

\begin{corollary}\label{denseframe}
  Let $\ell,t \in \posi$ with $\ell \ge 2$x. If $M \in \cU(\ell)$ is a simple frame matroid for which $|M| > (2d(t) + 1)(\ell-1) r(M)$, then $M$ has an $M(K_t)$-minor. 
\end{corollary}
\begin{proof}
  Let $c = d(t) + \tfrac{1}{2}$. We have $(2d(t)+1)(\ell-1) \ge 2c(\ell-1) + 2 - \ell$, so by Lemma~\ref{denseframetographic}, the matroid $M$ has a simple graphic minor $N$ for which $|N| > c r(N)$. Let $r = r(N)$, and $G = (V,E,\iota)$ be a graph on $r+1$ vertices for which $N = M(G)$. Since the number of vertices is at least one more than the average degree, we have $r + 1 = |V| \ge \frac{2|E|}{|V|} +1 > \tfrac{2cr}{r+1}+1$, which implies that $r \ge 2c-1 = 2d(t)$. Therefore 
  \[|E| = |N| > cr = (c-d(t))r -d(t) + d(t)(r+1) \ge d(t) (r+1) = d(t) |V|.\]
  It follows that $G$ has an $K_t$-minor, so $M$ has a $M(K_t)$-minor. 
\end{proof}

For a set $B$, a matroid $M$ is a \emph{$B$-clique} if $M$ is framed by $B$, and every pair of elements of $B$ is contained in a triangle of $M$. These matroids are dense, and we can use this density with the last corollary to find graphic clique minors of $B$-cliques. 

\begin{corollary}\label{cliquetocomplete}
  Let $\ell,t \in \posi$ with $\ell \ge 2$. If $M$ is a $B$-clique for which $M \in \cU(\ell)$ and $r(M) \ge (4d(t) + 2)(\ell-1)$, then $M$ has an $M(K_t)$-minor. 
\end{corollary}
\begin{proof}
  Let $r = r(M)$. Since $M$ is a $B$-clique, we have 
  $|M| \ge r + \tbinom{r}{2} > \tfrac{r^2}{2} \ge \frac{(4d(t) + 2) (\ell-1)}{2} r = (2d(t)+1)(\ell-1) r$,  
  so the result holds by Corollary~\ref{denseframe}.
\end{proof}

Recall that, for a graph $G$ on $n$ vertices, and a finite group $\Gamma$, the biased graph $G^{\Gamma}$ comes from a gain on the graph $G^{|\Gamma|}$, which has a loop at each vertex and $|\Gamma|$ edges between each pair of vertices. 
We can now restate and prove Theorem~\ref{mainframeclique}.

\begin{theorem}\label{framecliquetech}
  Let $\ell,t \in \posi$ with $\ell \ge 2$, and $\cM$ denote the class of frame matroids with no $U_{2,\ell+2}$-minor and no $M(K_t)$-minor. There exists $n_0 \in \posi$ such that \[(\ell-1)(\alpha + o_t(1)) t \sqrt{\log t} \cdot n \le h_{\cM}(n) \le 2(\ell-1)(\alpha + o_t(1)) t \sqrt{\log t} \cdot n\]
  for all $n \ge n_0$. 
\end{theorem}
\begin{proof}
  We first prove a lower bound.
  \begin{claim}
    There exists $n_0 \in \posi$ such that $h_{\cM}(n) > (\ell-1) d(t-1) n$ for all $n \ge n_0$. 
  \end{claim}
  \begin{subproof}
    Let $\Gamma$ be a group of order $\ell-1$. Let $n_0$ be given by Lemma~\ref{dlimit} for $t-1$ so that, for each $n \ge n_0$ there is a simple $K_{t-1}$-minor-free graph $G_n$ on $n$ vertices and more than $(d(t-1)-(\ell-1)^{-1}) n$ edges. For each $n \ge n_0$, let $M_n = \FM(G_n^{\Gamma})$. Since $M_n$ is $\Gamma$-frame, we have $M_n \in \cU(|\Gamma|+1) = \cU(\ell)$. 
    
    If $M_n$ had an $M(K_t)$-minor, then we would have $M(K_t) = \FM(H,\cB)$ for some minor $(H,\cB)$ of $G_n^{\Gamma}$. Let $H'$ be obtained from $H$ by deleting all loops. By Lemma~\ref{cliqueframerep}, the graph $H'$ is either $K_{t-1}$ or $K_t$. It follows from Lemma~\ref{looplessbiasedminor} that $K_{t-1}$ is a minor of the underlying graph $G_n^{|\Gamma|}$ of $G_n^{\Gamma}$. The graph $K_{t-1}$ is simple, so is also a minor of $G_n$, contrary to the choice of $G_n$. Therefore $M_n$ has no $M(K_t)$-minor, so $M_n \in \cM$. 
    
    Since $|M_n| = |\Gamma||E(G_n)| + n > (\ell-1)(d(t-1)-(\ell-1)^{-1})n + n = d(t-1)n$ and $r(M_n) = n$, the claim follows. 
  \end{subproof}
  
  By Corollary~\ref{denseframe}, we have $h_{\cM}(n) \le (2d(t)+1) (\ell-1) d(t) n$ for all $n$. The fact that $d(t) = (\alpha + o_t(1))t \sqrt{\log t}$ implies that $d(t-1) = (1 + o_t(1))d(t)$ and $2d(t)+1 = 2(1+o_t(1))d(t)$; the theorem now follows. 
\end{proof}

\section{Towers}\label{towersection}

For each $n \in \posi$, let $\cS(n)$ be the collection of nonempty subsets of $[n]$. When referring to sets in $\cS(n)$, we omit set brackets and union symbols where there is no ambiguity, writing $i,ij,Si$ for the sets $\{i\},\{i,j\}$ and $S \cup \{i\}$ respectively, where $i,j \in [n]$ and $S \in \cS(n)$.

We now define an auxiliary object called a \emph{tower} in a matroid. These are defined recursively, and are essentially the structures that arise from contractions that collapse many points together. 

\begin{definition}\label{tdef}
    A \emph{$1$-tower} in a matroid $M$ is a $1$-tuple $(e_1)$, where $e_1$ is a nonloop of $M$.  For $n \ge 2$, recursively define an \emph{$n$-tower} in a matroid $M$ to be a family $\bT = (e_X : X \in \cS(n))$ of elements of $M$ such that 
    \begin{enumerate}
        \item\label{dpar} $e_X \para_{M \con e_n} e_{Xn}$ for each $X \in \cS(n-1)$, 
        \item\label{nontriv} there exists $X \in \cS(n-1)$ for which $e_X \not\para_M e_{Xn}$, and 
        \item\label{drec} Each of $\bT^0 = (e_X : X \in \cS(n-1))$ and $\bT^1 = (e_{Xn} : X \in \cS(n-1))$ is an $(n-1)$-tower of both $M$ and $M \con e_n$. 
    \end{enumerate}
\end{definition}

Note that the function $X \mapsto e_X$ is not necessarily injective. Conditions (\ref{dpar}) and (\ref{nontriv}) together imply that $e_n$ is not a loop, and it is implicit in (\ref{dpar}) that $e_X$ is a nonloop for all other $X$. 

Given an $n$-tower $\bT$ of $M$, write $E(\bT)$ for $\{e_X : X \in \cS(n)\}$ and $M | \bT$ for $M | E(\bT)$. For each $X \subseteq [n]$, let $J^{\bT}_X = \{e_i : i \in X\}$. Where there is no ambiguity, we write $J_X$ for $J^{\bT}_X$. The elements of $J_{[n]} = \{e_1, \dotsc, e_n\}$ are the \emph{joints} of $\bT$. 

It is easy to see that a $2$-tower is a triangle $(e_1,e_{12},e_2)$. In general, $k$-towers can have complicated structures that don't depend only on $k$. Our eventual goal is to show that they contain large  (polynomial-sized) complete-graphic matroids as minors. Our next few lemmas establish a number of their basic properties. 

\begin{lemma}\label{tfacts}
    Let $n \in \posi$, let $\bT = (e_S : S \in \cS(n))$ be an $n$-tower in a matroid $M$, and let $J = J^{\bT}$. If $X \subseteq [n-1]$ and $k \in \posi$ satisfy $\max(X) < k \le n$, then 
    \begin{enumerate}[(i)]
        \item\label{find} $J_{Xk}$ is independent in $M$,
        \item\label{find2} $\{e_{ik} : i \in X\} \cup \{e_k\}$ is independent in $M$,
        \item\label{fjcl} $e_{Xk} \in \cl_M(J_{Xk})$,
        \item\label{ftri} if $X \ne \es$, then $e_{Xk} \in \cl_M(\{e_X,e_k\})$,
        \item\label{fspn} if $X \ne \es$, then $e_{Xk} \in \cl_M(\{e_{ik} : i \in X\})$.
    \end{enumerate}
\end{lemma}
\begin{proof}
    If $X = \es$, then all the conlusions are trivial, and $n = 1$ implies $X = \es$. Assume, therefore, that $X \ne \es$, that $n \ge 2$, and (inductively) that the lemma holds for smaller $n$. If $k < n$, then the $(n-1)$-tower $\bT^0$ of $M$ satisfies $e^{\bT^0}_Y = e_Y$ and $J^{\bT^0}_Y = J_Y$ for all $Y \subseteq Xk$, and so all conclusions hold by applying induction to $\bT^0$. We may assume, therefore, that $k = n$. Since $X \ne \es$, we have $X = X_0k_0$, where $X_0 \subseteq [n-2]$ and $\max(X) = k_0 \le 
    n-1$. Let $K_X = \{e_{ik} : i \in X\}$. 

    \begin{claim}\label{nspan}
        Neither $E(\bT^0)$ nor $E(\bT^1)$ spans $e_n$ in $M$.
    \end{claim}
    \begin{subproof}
        Let $d \in \{0,1\}$. By (\ref{find}) of the induction hypothesis applied to the tower $\bT^d$ of $M \con e_n$, we see that $J^d = J^{\bT^d}_{[n-1]}$ is an independent $(n-1)$-set in $M \con e_n$, and by (\ref{fjcl}) applied to the tower $\bT^d$ of $M$, we see that for each $Y \in \cS(n-1)$ we have $e^{\bT^d}_Y \in \cl_M(J^{\bT^d}_Y) \subseteq \cl_M(J^d)$, from which it follows that $J^d$ spans $E(\bT^d)$ in $M$. Since $J^d$ is independent in $M \con e_n$; it follows that $e_n \notin \cl_M(E(\bT^d))$, giving the claim. 
    \end{subproof}

    By applying (\ref{find}) and (\ref{find2}) inductively to $X_0$ and $k_0$ in the towers $\bT^0$ and $\bT^1$ of $M$, we get that $J^{\bT^0}_{X_0k_0} = J_X$ and $J^{\bT^1}_{X_0k_0} = K_X$ are independent in $M$. Now $J_X \subseteq E(\bT^0)$ and $K_X \subseteq E(\bT^1)$ while $e_k = e_n$, so (\ref{find}) and (\ref{find2}) both follow from~\ref{nspan}. 

    Since $k = n$, conclusion (\ref{ftri}) follows directly from~\ref{tdef}(\ref{dpar}). To see (\ref{fjcl}), by induction applied to $\bT^0$ we have $e_{X} = e_{X_0k_0} = e^{\bT^0}_{X_0k_0} \in \cl_M(J^{\bT_0}_{X_0k_0}) = \cl_M(J_X)$ and by (\ref{ftri}) we therefore have $e_{Xn} \in \cl_M(\{e_X,e_n\}) \subseteq \cl_M(J_X \cup \{e_n\}) = \cl_M(J_{Xn})$. Finally, to prove (\ref{fspn}), apply (\ref{fjcl}) of the inductive hypothesis to the $(n-1)$-tower $\bT^1$ of $M$ to get $e_{Xn} = e^{\bT^1}_X \in \cl_M (J^{\bT_1}_X) = \cl_M(K)$. 
\end{proof} 

The next corollary is immediate by applying (\ref{find}) and (\ref{fjcl}) to $X - \{k\}$ and $k$, where $k = \max(X)$.

\begin{corollary}\label{cjoint}
    If $n \in \posi$ and $\bT$ is an an $n$-tower in a matroid $M$, then for each $X \in \cS(n)$, the set $J^{\bT}_X$ is independent and spans $e^{\bT}_X$ in $M$. 
\end{corollary}

\begin{corollary}\label{cspan}
    If $n \in \posi$ and $\bT$ is an an $n$-tower in a matroid $M$, then $M | \bT$ has rank $n$, and has $J^{\bT}_{[n]}$ as a basis. 
\end{corollary}
\begin{proof}
    Corollary~\ref{cspan} implies that $J^{\bT}_{[n]}$ is an independent $n$-set in $M | \bT$, so it suffices to show that $J_{[n]}$ spans $M | \bT$. For each $e = e_X$ in $E(\bT)$, Corollary~\ref{cspan} gives $e_X \in \cl_M(J^{\bT}_X) \subseteq J^{\bT}_{[n]}$, as required. 
\end{proof}

The next corollary follows by applying~\ref{tdef}(\ref{dpar}) and Lemma~\ref{tfacts}(\ref{ftri}). 

\begin{corollary}\label{ttri} 
    Let $n \in \posi$, and let $\bT = (e_X : X \in \cS(n))$ be an $n$-tower in a matroid $M$. If $X \in \cS(n-1)$ and $k \in [n]$ with $\max(X) < k$, then either $M | \{e_X,e_{Xk},e_k\}$ is a triangle of $M$, or $e_{Xk} \approx_M e_X \not\approx_M e_k$.
\end{corollary}
\begin{lemma}\label{nontrivone}
    Let $n \in \posi$ and $\bT = (e_X : X \in \cS(n))$ be an $n$-tower in a matroid $M$. For all $1 < k \le n$, there exists $i < k$ for which $\{e_i,e_{ik},e_k\}$ is a triangle of $M$.
\end{lemma}
\begin{proof}
    If this is not the case, then by Corollary~\ref{ttri}, we have $e_i \para_M e_{ik}$ for all $i \in [k-1]$. Let $X \in \cS(k-1)$. By Lemma~\ref{tfacts}(\ref{fspn}), we have therefore $e_{Xk} \in \cl_M(\{e_{ik},k \in X\}) = \cl_M(J_X)$. We also have $e_X \in \cl_M(J_X)$ by Lemma~\ref{cjoint}. If $\{e_X,e_{Xk},e_k\}$ is a triangle, then this implies that $e_k \in \cl_M(\{e_{Xk},e_X\}) \subseteq \cl_M(J_X)$, contradicting Lemma~\ref{tfacts}(\ref{find}). Therefore, by Corollary~\ref{ttri}, we have $e_{X} \para_M e_{Xk}$ for all $X$, which contradicts \ref{tdef}(\ref{nontriv}).
\end{proof}

Finally, we show that a single contraction cannot destroy a large number of towers in a matroid in $\cU(\ell)$. If $\bT$ and $\bT'$ are $n$-towers in a matroid $M$, then call $\bT$ and $\bT'$ \emph{equivalent in $M$}, and write $\bT \para_M \bT'$, if $e^{\bT}_X \para_M e^{\bT'}_X$ for all $X \in \cS(n)$. If $M$ is simple, then $\bT \para_M \bT'$ implies $\bT = \bT'$. In general, equivalence classes of $n$-towers in $M$ are in bijection with $n$-towers in the simplification $\si(M)$. The tower-related concepts that we consider are all invariant under equivalence, so we can typically assume simplicity of $M$ with no loss of generality. 

\begin{lemma}\label{overcount}
    Let $n,\ell \in \posi$, let $M$ be a matroid with no $U_{2,\ell+2}$-restriction, and let $x \in E(M)$. If $\bT$ is an $n$-tower in both $M$ and $M \con x$, then there are at most $\ell^n$ inequivalent $n$-towers of $M$ that are equivalent to $\bT$ in $M \con x$. 
\end{lemma}
\begin{proof}
    Let $\bT = (e_S : S \in \cS(n))$. If $x$ is a loop, then the statement is trivial; we may thus assume that $M$ is simple. Let $\cT$ be the set of $n$-towers of $M$ that are $n$-towers equivalent to $\bT$ in $M \con x$. For each $X \in \cS(n)$, let $L_X = \cl_M(\{e_X,x\}) - \{x\}$, and let $\cL$ be the cartesian product of $L_{\{i\}}$ over all $i \in [n]$.  Since $M$ has no $U_{2,\ell+2}$-restriction and $r_M(L_X \cup \{x\}) = 2$, we have $|L_X| \le \ell$ for each $X$, so $|\cL| \le \ell^n$.
    
    For each $\bT' \in \cT$, the fact that $\bT' \para_{M \con x} \bT$ implies $e^{\bT'}_X \para_{M \con x} e_X$, so $e^{\bT'}_X \in \cl_M(\{e_X,x\}) - \{x\} = L_X$. We argue that $|\cT| \le \ell^{n}$ by showing that the map $\varphi : \bT' \mapsto (e_{1}^{\bT'}, e_{2}^{\bT'}, \dotsc, e_{n}^{\bT'})$ is an injection from $\cT$ to $\cL$. 

    The fact that the range of $\varphi$ is contained in $\cL$ follows from the fact that $e_i^{\bT'} \in L_i$ for each $\bT' \in \cT$ and $i \in [n]$. To see that $\varphi$ is injective, let $\bT_1,\bT_2 \in \cT$ with $\varphi(\bT_1) = \varphi(\bT_2)$. Then for each $X \in \cS(n)$ and $d \in \{1,2\}$, we have $e^{\bT_d}_X \in L_X$, and also $e^{\bT_d}_X \in \cl_M(J^{\bT_d}_X)$ by Lemma~\ref{tfacts}(\ref{fjcl}). The set $L_X \cup \{x\}$ has rank $2$ in $M$ and contains $x$, which is not spanned by $J^{\bT_d}_X$ because $r_M(E(\bT^d)) = n = r_{M \con x}(E(\bT^d))$. Therefore $L_X \cap \cl_M(J^{\bT_d}_X)$ has rank at most $1$ in $M$. Both $e^{\bT_1}_X$ and $e^{\bT_2}_X$ lie in this rank-one set, so by simplicity they are equal. This holds for all $X$, so $\bT_1 = \bT_2$, completing the proof of injectivity. 
\end{proof}

\section{Building Towers}\label{buildsection}

Our goal here is to show that, if $M \in \cU(\ell)$ with $\elem(M) \ge Cr(M)$, then some minor of $M$ contains a tower of order $t = \log(\sqrt{C})$. We do this by showing inductively that, for each $i \in \{1, \dotsc, t\}$, there is a minor $M_i$ of $M$ whose number of $i$-towers exceeds its number of $(i-1)$-towers by a factor of $C_i$, where $C_0 = C$, and $C_i$ decays exponentially with $i$. 

For a matroid $M$, write $w_0(M) = r(M)$, and for $n \in \posi$, write $w_n(M)$ for the number of inequivalent $n$-towers in $M$, or equivalently the number of $n$-towers in the simplification $\si(M)$. 

\begin{lemma}\label{grow}
    Let $\ell,t \in \posi$ with $\ell \ge 2$ and let $\alpha \in \bR$. If $M \in \cU(\ell)$ satisfies $w_t(M) > \ell^{t-1}(\alpha + \ell^t)w_{t-1}(M)$, then $M$ has a minor $M_0$ with $w_{t+1}(M_0) > \alpha w_t(M_0)$. 
\end{lemma}
\begin{proof}
    Set $\beta = \ell^{t-1}(\alpha + \ell^t)$. 
    Let $M_0$ be a minimal minor of $M$ for which $w_t(M_0) > \beta w_{t-1}(M_0)$; we show that $M_0$ satisfies the lemma. Note that $M_0$ is simple, since simplifying changes neither $w_t$ nor $w_{t-1}$.
    For each $x \in M_0$ and $i \in \bN$, let $A_i(x)$ be the set of $i$-towers $\bT$ of $M_0$ for which $x \in \cl_M(E(\bT))$. The $i$-towers of $M_0$ outside $A_i(x)$ are all $i$-towers of $M_0 \con x$; let $\cE_i(x)$ be the collection of equivalence classes of towers in $M_0 \con x$, so $|\cE_i(x)| = w_i(M_0 \con x)$. We make several claims about the $A_i$ and $\cE_i$. The first simply follows from $\cE_i$ being a partition. 
    \begin{claim}\label{epart} $w_i(M_0) = |A_i(x)| + \sum_{C \in \cE_i(x)} |C|$ for each $x \in M_0$ and $i \in \bN$. \end{claim}
    The next is proved by reversing the order of a summation. 
    \begin{claim}\label{suma}
        $\sum_{x \in M_0} |A_i(x)| \le \ell^i w_i(M_0)$ for each $x \in M_0$ and $i \in \bN$.
    \end{claim}
    \begin{subproof}
        Since $M_0 \in \cU(\ell)$, each $i$-tower $\bT$ has rank $i$ and so by Theorem~\ref{kung} spans at most $\ell^i$ points of $M_0$, so is in $A_i(x)$ for at most $\ell^i$ different $x$. Therefore $\sum_{x \in M_0} |A_i(x)| = |\{(x,\bT) : \bT \in A_i(x)\}| \le \ell^i w_i(M_0)$.  
    \end{subproof}
    The next claim relates the $\cE_i$ to $(i+1)$-towers. 
    \begin{claim}\label{nexti}
        $\sum_{x \in M_0} \sum_{C \in \cE_i(x)}|C|(|C|-1) = w_{i+1}(M_0)$ for each $x \in M_0$ and $i \in \bN$. 
    \end{claim}
    \begin{subproof}
        The given sum is equal to the number of triples $(x,\bT_0,\bT_1)$ for which $x \in M_0$ and $\bT_0,\bT_1$ are distinct elements of some class in $\cE_i(x)$, or in other words, triples $(x,\bT_0,\bT_1)$ such that $\bT_0$ and $\bT_1$ are inequivalent $i$-towers in $M$ but are equivalent $i$-towers in $M \con x$. 
        
        Given such a triple, construct an $(i+1)$-tower $\bT = (e_X : X \in \cS(i+1))$ by setting $e_{i+1} = x$ and for all $X \in \cS(i)$, setting $e_X = e^{\bT_0}_X$ and $e_{X(i+1)} = e^{\bT_1}_X$. It is routine to verify from~\ref{tdef} that $\bT$ is a tower; condition (\ref{dpar}) follows from $\bT_0 \para_{M \con x} \bT_1$, while (\ref{nontriv}) follows from the simplicity of $M_0$ and $\bT_0 \ne \bT_1$, and (\ref{drec}) holds for $\bT^0 = \bT_0$ and $\bT^1 = \bT_1$. So each triple determines an $(i+1)$-tower; it is easy to see that this correspondence is bijective (each $(i+1)$-tower $\bT$ gives the triple $(e^\bT_{i+1}, \bT^0, \bT^1)$ where $\bT^0,\bT^1$ are the towers from \ref{tdef}(\ref{drec})), so the claim holds. 
    \end{subproof}
    For each $x \in M_0$ and $C \in \cE_i(x)$, we have $1 \le |C| \le \ell^i$ by Lemma~\ref{overcount}. Applying these two inequalities to the first $|C|$ in~\ref{nexti} and rearranging gives the following. 
    \begin{claim}\label{ssumbounds}
        $\ell^{-i}w_{i+1}(M_0) \le \sum_{x \in M_0}\sum_{C \in \cE_i(x)}(|C|-1) \le w_{i+1}(M_0)$.
    \end{claim}
    Let $\Delta_i = \sum_{x \in M_0}(w_i(M_0) - w_i(M_0 \con x))$ for each $i \in \nni$. Note that $\Delta_0 = |M_0|$. 
    \begin{claim}\label{wdelta}
        $\Delta_i - \ell^i w_i(M_0) \le w_{i+1}(M_0) \le \ell^i \Delta_i$ for each $i \in \nni$.
    \end{claim}
    \begin{subproof}
        This is immediate for $i = 0$, because $w_1(M_0) = \Delta_0$. Otherwise, since $|\cE_i(x)| = w_{i}(M_0 \con x)$ and $\sum_{C \in \cE_i(x)}|C| = w_i(M_0)- |A_i(x)|$ for each $x$, we have 
        \begin{align*} \Delta_i &= \sum_{x \in M_0}\left(\left(\sum_{C \in \cE_i(x)} |C|\right)  + |A_i(x)|- |\cE_i(x)|\right)  \\
            &= \sum_{x \in M_0} |A_i(x)| +  \sum_{x \in M_0}\sum_{C \in \cE_i(x)} (|C| - 1) \\ 
            &\le \ell^i w_i(M_0) + w_{i+1}(M_0),\end{align*}
        and
        $ \Delta_i \ge \sum_{x \in M_0}\sum_{C \in \cE_i(x)}(|C| - 1) \ge \ell^{-i} w_{i+1}(M_0)$. by~\ref{ssumbounds}.
    \end{subproof}

    \begin{claim}\label{deltat}
        $\Delta_t > \beta \ell^{1-t}w_t(M_0)$. 
    \end{claim}
    \begin{subproof}
        By the choice of $M_0$, we have $w_t(M_0) > \beta w_{t-1}(M_0)$ and also  $w_t(M_0 \con x) \le \beta w_{t-1}(M_0 \con x)$ for each $x$. Hence $w_{t}(M_0) - w_t(M_0 \con x) > \beta (w_{t-1}(M_0) - w_{t-1}(M_0 \con x))$; summing this over all $x$ gives $\Delta_t > \beta \Delta_{t-1}$. Now by the upper bound in \ref{wdelta}, we have $\Delta_{t-1} \ge \ell^{1-t} w_t(M_0)$; combining these inequalities gives the claim.  
    \end{subproof}
    Using the lower bound in $\ref{wdelta}$ and then~\ref{deltat}, we get
    \[ w_{t+1}(M_0) \ge \Delta_t - \ell^t w_t(M_0) >\beta \ell^{1-t}w_t(M_0) - \ell^t w_t(M_0) = (\beta \ell^{1-t} - \ell^t) w_t(M_0).\]
    Since $\beta \ell^{1-t} - \ell^t = \alpha$, we have $w_{t+1}(M_0) > \alpha w_t(M_0)$ as required. 
\end{proof}

\begin{lemma}\label{gettower}
    Let $\ell,t \in \posi$ and $M \in \cU(\ell)$. If $\elem(M) > \ell^{\binom{t+1}{2}} r(M)$, then there is a $t$-tower in some minor of $M$. 
\end{lemma}
\begin{proof} 
    This is vacuous for $\ell = 1$, since $M \in \cU(1)$ implies $\elem(M) \le r(M)$. Otherwise, let 
    $c_k = \ell^{\binom{t+1}{2} - \binom{k}{2}}$ for all $k \in \posi$.  We show by induction that, for all $k \in [t]$, there is a minor $M_k$ of $M$ for which $w_k(M_k) > c_k w_{k-1}(M_k)$. Putting $k=t$ gives the result for the minor $M_t$, since it gives $w_t(M_t) > c_t w_{t-1}(M_t) > 0$. 
    
    Since $w_1(M) = \elem(M) > c_1 r(M) = w_0(M)$, we can take $M_1 = M$. 
    If the required $M_k$ exists for some $k \in [t-1]$, then since $c_k \ge c_t = \ell^t \ge \ell^k$ and $\ell \ge 2$, we have 
    \begin{align*}
        \ell^{k-1}(c_{k+1} + \ell^k) &\le 2\ell^{k-1} c_{k+1} 
        \le \ell \cdot \ell^{\binom{t+1}{2}-\binom{k+1}{2} + k-1}  
        = \ell^{\binom{t+1}{2} - \binom{k}{2}} = c_k.
    \end{align*}
    By Lemma~\ref{grow} and the fact that $w_k(M_k) > c_k w_k(M_{k-1}) \ge \ell^{k-1}$, there is a minor $M_{k+1}$ of $M_k$ with $w_{k+1}(M_{k+1}) > c_{k+1} w_k (M_{k+1})$; this completes the induction. 
\end{proof}

\section{Exploiting Towers}\label{exploitsection}

The remainder of our proof is devoted to showing that, for $h$ of order $t^2 \log t$, a matroid containing an $h$-tower must contain an $M(K_t)$-minor. We do this by associating each tower with a particular digraph $G$, whose vertices are positive integers, in which every edge is directed upwards. A \emph{digraph} is a pair $G = (V,A)$, where $V$ is a finite subset of $\posi$, and $A \subseteq V \times V$ satisfies $v_0 < v_1$ for all $(v_0,v_1) \in A$. (Note that this definition is nonstandard.)

Given $S \subseteq V$, write $G[S]$ for the induced subdigraph $(S, A \cap (S \times S))$. A digraph $(\{v_1, \dotsc, v_n\},A)$ with $v_1 < \dotsc < v_n$ 
is a \emph{path from $v_1$ to $v_n$} if $A = \{(v_i,v_{i+1}) : 1 \le i < n\}$. A digraph $G = (V,A)$ is \emph{connected} if $V \ne \es$ and, for all $x \in V$, some subdigraph of $G$ is a path from $\min(V)$ to $x$. 

For each $n$-tower $\bT = (e_X : X \in \cS(n))$ in $M$, let $G(\bT)$ be the digraph on vertex set $[n]$ in which $(i,j)$ is an arc if and only if $i < j$ and the set $\{e_i,e_{ij},e_j\}$ is a triangle of $M$. By Lemma~\ref{ttri}, if $i < j$ and $(i,j)$ is not an arc, then $e_i \para_M e_{ij}$. By Lemma~\ref{nontrivone}, for each $j > 1$ there is some arc $(i,j)$ with $i < j$; an inductive argument gives the following. 

\begin{lemma}\label{tgraph}
    Let $n \in \posi$, and $M$ be a matroid. If $\bT$ is an $n$-tower in $M$, then $G[\bT]$ is connected. 
\end{lemma}

A \emph{tree} of an $n$-tower $\bT$ in a matroid $M$ is a set $S \in \cS(n)$ such that each $s \in S$ with $s > \min(S)$ has in-degree exactly $1$ in the subdigraph $G(\bT)[S]$. For a tree $S$, we know by Lemma~\ref{tfacts}(\ref{fjcl}) and (\ref{find}) that $e_S$ is spanned by the independent $|S|$-set $J_S$; this next lemma shows that $J_S$ is a minimal set with this property. 

\begin{lemma}\label{treecct}
    Let $n \in \posi$, and $\bT = (e_X : X \in \cS(n))$ be an $n$-tower in a matroid $M$. Let $S \in \cS(n)$ and $k$ satisfy $\max(S) < k \le n$. If $Sk$ is a tree of $\bT$, then both $\{e_S,e_{Sk},e_k\}$ and $J^{\bT}_{Sk} \cup \{e_{Sk}\}$ are circuits of $M$.
\end{lemma}
\begin{proof}
    
    We may assume that $M$ is simple. Write $J,G$ for $J^{\bT}$ and $G(\bT)$ respectively. 

    If $|S| = 1$, then the definition of $G$ and a tree implies that $\{e_S,e_{Sk},e_k\} = J_S \cup \{e_S\}$ is a triangle, as required. Suppose, therefore, that $|S| \ge 2$ and that the result holds for smaller $S$. It is clear that $S$ is also a tree of $\bT$. By induction applied to $(S',k') = (S - \{\max(S)\},\max(S))$, the the set $J_S \cup \{e_S\}$ is a circuit of $M$.
    
    Since $Sk$ is a tree, there is some unique $d \in S$ for which $(d,k)$ is an arc, so $\{e_d,e_{dk},e_k\}$ is a triangle of $M$. We argue the first part of the conclusion.

    \begin{claim}
        $\{e_S,e_{Sk},e_k\}$ is a triangle of $M$. 
    \end{claim}
    \begin{subproof}
        Suppose not. By Lemma~\ref{ttri} and simplicity, we therefore have $e_S = e_{Sk}$, and also $e_i = e_{ik}$ for all $i \in S - \{d\}$. Let $H = J_{S - \{d\}}$. Since
        $\{e_d,e_{dk},e_k\}$ is a triangle of $M \del H$, we have $r_{M \con H}(\{e_d,e_{dk}\}) = r_{M \con H}(\{e_d,e_k,e_{dk}\}) \ge r_{M \con H}(\{e_d,e_k\})$. Lemma~\ref{ttri}(\ref{find}) gives that $H \cup \{e_d,e_k\}$ is independent in $M$, so we have $r_{M \con H}(\{e_d,e_{dk}\}) \ge r_{M \con H}(\{e_d,e_k\}) = 2$, and thus $(M \con H)|\{e_{dk},e_d\} \cong U_{2,2}$. 
         Since $H \cup \{e_d,e_S\} = J_S \cup \{e_S\}$ is a circuit of $M$, we have $(M \con H)|\{e_d,e_S\} \cong U_{1,2}$. Combining these gives $e_S \notin \cl_{M \con H}(\{e_{dk}\})$. 
        However, Lemma~\ref{tfacts}(\ref{fspn}) yields 
        \[e_S = e_{Sk} \in \cl_M(\{e_{ik} : i \in S\}) = \cl_M(H \cup \{e_{dk}\}),\] so $e_S \in \cl_{M \con H}(\{e_{dk}\})$, a contradiction.
    \end{subproof}

    By the inductive hypothesis and the claim, both the sets $C_1 = J_{S} \cup \{e_S\}$ and $C_2 = \{e_S,e_{Sk},e_k\}$ are circuits of $M$. Lemma~\ref{tfacts}(\ref{find}), gives $e_k \notin \cl_M(J_S)$, so $C_1$ does not span $e_k$. Thus $C_1 \cap C_2 = \{e_S\}$, and $M|(C_1 \cup C_2 - \{e_S\})$ is a series extension of $M|C_1$, so $C_1 \cup C_2 - \{e_S\} = J_{Sk} \cup \{e_{Sk}\}$ is a circuit of $M$, as required.
\end{proof}

\begin{lemma}\label{pathclique}
    Let $n \in \posi$, and $\bT$ be an $n$-tower in a matroid $M$. If $X \in \cS(n)$ and $G(\bT)[X]$ is a path, then $M$ has an $M(K_{|X|+1})$-restriction.
\end{lemma}
\begin{proof}
    We may assume that $M$ is simple. Let $X = \{x_1, \dotsc, x_t\}$ with $x_1 < \dotsc < x_t$. For each $A \in \cS(t)$, write $X(A) = \{x_i : i \in S\}$, and $f(A) = e^{\bT}_{X(A)}$. (One can in fact show that $(f(A) : A \in \cS(t))$ is a $t$-tower in $M$, but proving this is not needed.)
    
    Let $K$ be the complete graph with vertex set $[t+1]$ in which, for all $1 \le i < j \le t$, the edge from $i$ to $j$ is formally identified with the interval $[i,j-1]$. In particular, the function $f$ maps each edge of $K$ to an element of $M$. For each $s \in [t]$, write $K^s$ for the subgraph of $K$ induced by $[s+1]$, and $M^s$ for the restriction of $M$ to the image $f(E(K^s))$. We show by induction on $s$ that, for all $s \in [t]$, the restriction of $f$ to $E(K^s)$ is an isomorphism from the graphic matroid $M(K^s)$ to $M^s$. This is immediate when $s = 1$, since $K^s$ has just one edge, and towers comprise nonloops. 

    Let $Z_0 = \{[i,s-1] : 1 \le i < s\}$ and $Z = Z_0 \cup \{\{s\}\}$; the set $Z$ is the edge-neighbourhood of the vertex $s$ in the graph $K^s$. 

    \begin{claim}\label{alltri}
        For every triangle $T$ of $K^s$ containing two edges in $Z$, the set $f(T)$ is a triangle of $M$.
    \end{claim}
    \begin{subproof}
        Each such $T$ either contains the edge $\{s\}$ or does not, so by the definition of $K$, either has the form $\{[i,s-1],[i,s],\{s\}\}$ for some $1 \le i < s_2$, or $\{[i,s],[j,s],[i,j-1]\}$ for some $1 \le i < j < s$.

        In the first case, since $X([s])$ is a path of $\bT$, it is a tree of $\bT$, so Lemma~\ref{treecct} with $(S,k) = (X([s-1]), x_s)$ gives that $\{f([i,s-1]), f([i,s]), f(s)\}$ is a triangle of $M$. In the second case, $T$ is a triangle of $K^{s-1} = K^s - \{s+1\}$; by induction on $s$, the set $f(T)$ is a triangle of $M$. 
    \end{subproof}

    In particular, this implies that no two edges in $Z$ are mapped to the same element of $M$, so $|f(Z)| = |Z|$.

    \begin{claim}
        $M^s$ is an $f(Z)$-clique. 
    \end{claim}
    \begin{subproof}
        We first show that $f(Z)$ is a basis for $M^s$. Since $Z_0$ is the neighbourhood of the vertex $s$ in $K^{s-1}$, it is a basis for $M(K^{s-1})$ and therefore $f(Z_0)$ is a basis for $M^{s-1}$ by induction. Let $J_0 = \{f(1), \dotsc, f(s-1)\}$; Lemma~\ref{tfacts}(\ref{find}) implies that $J_0 \cup f(s)$ is independent in $M$, while Corollary~\ref{cjoint} implies that $f(Z_0) \subseteq \cl_M(J_0)$ and $E(M^s) \subseteq \cl_M(J_0 \cup \{f(s)\})$; thus $f(Z) = f(Z_0) \cup \{f(s)\}$ is a basis for $M^s$. 

        Since every pair of edges of $Z$ is contained in a triangle of $K$, by~\ref{alltri} every pair of elements of $f(Z)$ spans a triangle in $M$. Moreover, every edge of $K^s$ is contained in a triangle containing two edges of $Z$, so every element of $M^s$ is contained in a triangle spanned by two elements of $Z$. Therefore $M^s$ is an $f(Z)$-clique. 
    \end{subproof}
    
    \begin{claim}
        $r(M^s \del f(Z)) < r(M^s)$. 
    \end{claim}
    \begin{subproof}
        It is enough to show that $f(s)$ is not spanned by $E(M^s \del f(Z))$ in $M$. By Lemma~\ref{tfacts}(\ref{find2}), to see this we can prove that, for each edge $e$ of $K^s \del Z$, the element $f(e)$ is spanned in $M$ by the set $Q = \{f(is) : 1 \le i < s\}$. Since $X$ is a path, we have $f(is) = f(i)$ for each $i < s-1$, so $Q = \{f(1), \dotsc, f(s-2), f(\{s-1,s\})\}$. 

        Let $e = [i,j]$ be an edge of $K^s \del Z$, so $i < j \ne s-1$ and $i < s$. If $j < s-1$, then by Corollary~\ref{cjoint} we have $f([i,j]) \in \cl_M(\{f(i), f(i+1), \dotsc, f(j)\}) \subseteq \cl_M(Q)$.
        If $j = s$, then since $i < s$, Lemma~\ref{tfacts}(\ref{fspn}) gives $f([i,s]) \in \cl_M(\{f(\{i,s\}) : i \in [s-1]\})$. This is contained in $\cl_M(Q)$ by the definition of $Q$. 
    \end{subproof}

    The last two claims and Lemma~\ref{cliquetographic} give that $M^s = M(K')$ for some complete graph $K'$ in which $f(Z)$ is the edge-neighbourhood of a vertex. We want to show that $f$ is an isomorphism from $M(K^s)$ to $M(K')$. This holds because both are complete graphs, the sets $Z$ and $f(Z)$ are edge-neighbourhoods in $K^s$ and $K'$ respectively, while~\ref{alltri} implies that every triangle of $M(K^s)$ containing two edges in $Z$ is mapped by $f$ to a triangle in $M(K')$.     
\end{proof}
For each $t,\ell \in \posi$, let $d_\ell(t) = \tfrac{1}{4}(t-1)$ if $\ell \le 2$, and let $d_{\ell}(t) = d(t)$ if $\ell \ge 3$. By Theorem~\ref{kostochka} we have $d_\ell(t) \le d(t)$ for all $\ell$ and $t \ge 3$. 

\begin{lemma}\label{getcliquefromtower}
    $\ell,t,n \in \posi$ with $\ell \ge 2$, $t \ge 3$, and $n \ge (t-2)(4d_{\ell}(t)+2)(\ell-1)$. Let $M \in \cU(\ell)$. If there is an $n$-tower in $M$, then $M$ has an $M(K_{t})$-minor. 
\end{lemma}
\begin{proof}
    Let $\bT = (e_X : X \in \cS(n))$ be a $t$-tower in $M$, and let $G = G(\bT)$ and $J = J^\bT$. We may assume that $M = M(\bT)$. For each $k \in [n]$, let $\pi(k)$ be the minimum $i \in [n]$ such that $(i,k)$ is an arc of $G$, or $\pi(k) = 1$ if no such arc exists. By the connectedness of $G$, each $k > 1$ has an in-neighbour, so $\pi(k) < k$ unless $k = 1$, and since $1$ has no in-neighbour, we have $\pi(1) = 1$. 

    For each $k \in [n]$, let $P_k = \cup_{i \in \nni} \pi^i(k)$, where the exponent denotes iteration of $\pi$. By the definition of $\pi$ and the connectedness of $G$, the set $P_k$ is a path from $1$ to $k$ in $G$. Let $C = \pi([n])$, and $L = [n] \del C$. Note that $P_k \subseteq C \cup \{k\}$ for each $k \in [n]$. 

    \begin{claim}\label{pathspart}
        $\cup_{x \in L} P_x = [n]$. 
    \end{claim}
    \begin{subproof}
        Suppose not, and let $k \in [n] - \cup_{x \in L} P_x$ be maximal. Clearly $k \notin L$ (otherwise $k \in P_k$), so there is some $k' > k$ with $\pi(k') = k$. By the maximality of $k$, some $P_x$ contains $k'$, and therefore contains $\pi(k') = k$, a contradiction. 
    \end{subproof}

    \begin{claim}
        $M \con J_C$ has a $J_L$-clique restriction. 
    \end{claim}
    \begin{subproof}
        By Corollary~\ref{cspan}, the set $J_{[n]} = J_C \cup J_L$ is a basis for $M$, so clearly $J_L$ is a basis for $M \con J_C$. It remains to show for any distinct $a,b \in L$, the elements $e_a,e_b$ span a triangle of $M \con J_C$. 

        Let $a,b \in L$ be distinct. Since $P_a,P_b$ are paths, the graph $G[P_a \cup P_b]$ is connected; let $S$ be minimal such that $\{a,b\} \subseteq S \subseteq C \cup \{a,b\}$ and $G[S]$ is connected. Let $Q_a,Q_b$ be the paths from $\min(S)$ to $a,b$ respectively in $G[S]$. If $Q_a$ and $Q_b$ intersect in some vertex $w$ other than $\min(S)$, then the subpaths of $Q_a,Q_b$ from $w$ provide a contradiction to the minimality of $S$. If there is an arc $(w,w')$ from some $w \ne \min(S)$ in one of $Q_a,Q_b$ to some $w'$ in the other, then $S - \{\min(S)\}$ contains the union of paths from $w$ to $a$ and $b$ respectively, contradicting minimality. Therefore the set of arcs of $G[S]$ is precisely the union of the sets of arcs of $Q_a$ and $Q_b$, so $S$ is a tree containing $a$ and $b$. 
        
        By Lemma~\ref{treecct}, the set $J_S \cup \{e_S\}$ is a circuit of $M$ containing $e_a,e_b,e_S$, so $T_0 = \{e_a,e_b,e_S\}$ is a triangle of $M' = M \con (J_S - \{e_a,e_b\})$. Since $a,b \notin C$, the set $J_C \cup \{e_a,e_b\}$ is independent in $M$ by Lemma~\ref{tfacts}(\ref{find}), so $\{e_a,e_b\}$ is independent in the contraction $M \con J_C = M' \con (J_C - J_S)$. Therefore $\{e_a,e_b\}$ spans a triangle in $M'$ and is independent in the contraction-minor $M \con J_C$ of $M'$, so $\{e_a,e_b\}$ spans a triangle of $M \con J_C$, as required. 
    \end{subproof}
    
    We assumed that $M$ has no $M(K_t)$-minor; if $\ell \ge 3$, then $d_{\ell}(t) = d(t)$, and it follows from Corollary~\ref{cliquetocomplete} that $|L| < (4d_{\ell}(t) + 2)(\ell-1)$. If $\ell = 2$, then $M$ is binary, so a $J_L$-clique implies an $M(K_{|L|+1})$-restriction, which also implies that $|L| < t < (4d_\ell(t) + 2)(\ell-1)$. 
    
    By~\ref{pathspart}, we get $n = |\cup_{x \in L} P_x| \le \sum_{x \in L} |P_x|$, so there exists $x \in L$ for which $|P_x| \ge \tfrac{n}{|L|}  > \frac{n}{(\ell-1)(4d_\ell(t) + 2)} \ge t-2$, so $|P_x| \ge t-1$. By Lemma~\ref{pathclique}, it follows that $M$ has an $M(K_{t})$-restriction, a contradiction. 
\end{proof}

\begin{theorem}\label{maintech}
    Let $\ell,t \in \bN$, and $M \in \cU(\ell)$. If $\elem(M) > \ell^{8 (\ell-1)^2 t^2 d_\ell(t)^2} r(M)$, then $M$ has an $M(K_{t})$-minor. 
\end{theorem}
\begin{proof}
  Since $\elem(M) > r(M)$, we know that $M$ has a $U_{2,3} \cong M(K_3)$-minor; we may therefore assume that $\ell \ge 2$ and that $t \ge 4$. Let $n_0 = \floor{4(\ell-1) t d(t)}-1$, and note that $\binom{n_0+1}{2} < 8 (\ell-1)^2 t^2 d(t)^2$, so $\elem(M) > \ell^{\binom{n_0+1}{2}}r(M)$. By Lemma~\ref{gettower}, there is an $n_0$-tower in some minor $M_0$ of $M$. If $\ell \ge 3$, then Theorem~\ref{kostochka} gives $d_{\ell}(t) = d(t) \ge t-2 \ge 2$; it follows that $4td(t) \ge (t-2)(4d(t) + 2) + 2$ and so $n_0 \ge (t-2)(4d(t) + 2)(\ell-1)$. If $\ell = 2$, then $d_{\ell}(t) = \tfrac{1}{4}(t-1)$, so $n_0 = t(t-1)-1 > (t-2)(t+1) = (t-2)(4d(t)+2)(\ell-1)$. 
  In both cases, Lemma~\ref{getcliquefromtower} gives that $M_0$ has an $M(K_t)$-minor. 
\end{proof}

We can now prove Theorem~\ref{mainbin}, which we restate in a slightly different form. 

\begin{theorem}\label{maintechbin}
  Let $t \in \posi$. If $M$ is a binary matroid with $\elem(M) > 2^{t^4 / 2} r(M)$, then $M$ has an $M(K_t)$-minor. 
\end{theorem}
\begin{proof}
  Since $M \in \cU(2)$, and $d_2(t) < \tfrac{t}{4}$, we have $\tfrac{1}{2}t^4 = 8(2-1)^2 t^2 (d_2(t))^2$, and the result follows from Theorem~\ref{maintech}. 
\end{proof}


For $t \ge 3$, Theorem~\ref{thomason} implies the inequality $d_\ell(t) \le d(t) = (\alpha + o(1))t\sqrt{\log t}$, where $\alpha \approx 0.319$. In combination with Theorem~\ref{maintech}, we have the following. 

\begin{theorem}\label{maintechasymp}
    Let $\ell,t \in \bN$, and let $M \in \cU(\ell)$. For all $C > 8\alpha^2$, if $t$ is sufficiently large and $\elem(M) > \ell^{C (\ell-1)^2 t^4 \log t} r(M)$, then $M$ has an $M(K_t)$-minor. 
\end{theorem}

Since $8\alpha^2 \approx 0.81 < 1$, this implies Theorem~\ref{main}. By Theorem~\ref{kostochka}, we have $d_{\ell}(t) \le d(t) \le 22 t \sqrt{\log t}$. This gives a version of Theorem~\ref{main} applying to all $t$. 

\begin{theorem}\label{mainabs}
  Let $\ell, t \in \bN$, let $C = 8 \cdot 22^2$, and let $M \in \cU(\ell)$. If $M$ satisfies $\elem(M) > \ell^{C (\ell-1)^2 t^4 \log t} r(M)$, then $M$ has an $M(K_t)$-minor. 
\end{theorem}

\section*{References}

\newcounter{refs}

\begin{list}{[\arabic{refs}]}%
{\usecounter{refs}\setlength{\leftmargin}{10mm}\setlength{\itemsep}{0mm}}

  \item\label{cdfp} R. Chen, M. DeVos, D. Funk and I. Pivotto,
  Graphical Representations of Graphic Frame Matroids. Graphs Comb. 31 (2015), 2075--2086.
  \item\label{dfp} M. DeVos, D. Funk, I. Pivotto, When does a biased graph come from a group labelling?, Adv. Appl. Math. 61 (2014), 1--18.   
  \item\label{g11} J. Geelen, Small cocircuits in matroids, Eur. J. Comb. 32 (2011), 795--801.
  \item\label{gn10} J. Geelen, P. Nelson, The number of points in a matroid with no $n$-point line as a minor, J. Comb. Theory. Ser. B 100 (2010), 625--630. 
  \item\label{gnw} J. Geelen, P. Nelson and Z. Walsh, Excluding a line from $\mathbb{C}$-representable matroids, Memoirs 
  \item\label{gw} J. Geelen, G. Whittle, Clique in dense $\GF(q)$-representable matroids, J. Comb. Theory. Ser. B 99 (2009), 1--8. 
  \item\label{heller} I. Heller, On linear systems with integral valued solutions. Pacific J. Math. 7 (1957), 1351--1364.
  \item\label{kk} J. Kahn, J.P.S. Kung. Varieties of combinatorial geometries. Trans. Amer. Math. Soc. 271 (1982), 485--499.
  \item\label{k82} A.V. Kostochka, The minimum Hadwiger number for graphs with a given mean degree of vertices, Metody Diskret. Analiz. 38 (1982), 37--58; AMS Translations 132 (1986), 15--32. 
  \item\label{kung87} J.P.S. Kung, Excluding the Cycle Geometries of the Kuratowski Graphs from Binary Geometries, Proc. London Math. Soc. s3-55 (1987), 209--242.
  \item\label{kungll} J.P.S. Kung, The long-line graph of a combinatorial geomery. I. Excluding $M(K_4)$ and the $(q+2)$-point line as minors, Quart. J. Math. Oxford 39 (1988), 223--234. 
  \item\label{kung93} J.P.S. Kung, Extremal matroid theory, in: Graph Structure Theory (Seattle WA, 1991), Contemporary Mathematics, 147, American Mathematical Society, Providence RI, 1993, pp.~21--61.
  \item\label{oxley} J. G. Oxley, Matroid Theory (2nd edition), Oxford University Press, New York, 2011.
  \item\label{t01} A. Thomason, The extremal function for complete minors, J. Combin. Theory. Ser. B 81 (2001), 318-338. 
  \item\label{zas1} T. Zaslavsky, The biased graphs whose matroids are binary, J. Comb. Theory. Ser. B 42 (1987), 337--347. 
  \item\label{zas94} T. Zaslavsky, Frame Matroids and Biased Graphs, Eur. J. Comb. 15 (1994), 303--307.
  \item\label{zas99} T. Zaslavsky, Biased Graphs. I. Bias, Balance and Gains, J. Comb. Theory. Ser. B 47 (1989), 32--52.

\end{list}

\end{document}